\documentclass[11pt]{amsart}

\usepackage[colorlinks=true,citecolor=cyan,backref=page]{hyperref}

\usepackage[square,sort&compress,comma,numbers]{natbib}
\usepackage{nicefrac,xcolor,upref}

\usepackage{amscd,amsthm,amsmath,amssymb,amsfonts}

\usepackage{mathrsfs}

\newtheorem{theorem}{Theorem}[section]

\newtheorem{definition}[theorem]{Definition} 
\newtheorem{problem}[theorem]{Problem}

\numberwithin{equation}{section}

\def\tmu{{\mu^{\times}}}

\def\Q{{\mathbb {Q}}}
\def\Z{{\mathbb Z}}  
\def\R{{\mathbb R}} 
\def\C{{\mathbb C}} 
 
\def\HH{{\mathbb H}} 
\def\eps{{\varepsilon}}

\def\cS{{\mathcal S}}

\def\ggeff{{\gg_{{\rm eff}}}\,}

\def\pp{{\mathfrak p}}

\def\mueff{{\mu_{\rm eff}}} 
\def\nueff{{\nu_{\rm eff}}} 
\def\vbeff{{v_{b, \rm eff}}}

\def\resp{{\rm resp., }}
\def\rme{{\rm e}} 
\def\rmi{{\rm i}}
\def\rmj{{\rm j}}
\def\rmk{{\rm k}}
\def\cN{{\mathcal N}}
\def\zprim{{\mathbb{Z}^2_{{\rm prim}}}}

\def\vv{{\rm v}}

\def\ovQ{{{\overline \Q}^*}}
\def\bG{{{\overline \Gamma}}}
\def\bfx{{\bf x}}
\def\bfy{{\bf y}}
\def\bfz{{\bf z}}

\def\tmu{{\mu^\times}}
\def\tmueff{{\mu^\times_{\rm eff}}}

\def\beq{\begin{equation}}
\def\eeq{\end{equation}}

\def\alphabar{{\overline {\alpha}}}
\def\tr{{\rm tr}}

\begin{document}

\title{$B'$}

\author{Yann Bugeaud}
\address{Universit\'e de Strasbourg, Math\'ematiques,
7, rue Ren\'e Descartes, 67084 Strasbourg  (France)} 
\address{Institut universitaire de France} 
\email{bugeaud@math.unistra.fr}

\begin{abstract}
Let $n \ge 2$ be an integer and $\alpha_1, \ldots, \alpha_n$ be non-zero algebraic numbers. 
%Let $\log \alpha_1, \ldots , \log \alpha_n$ be determinations of their logarithm and 
%assume that $\log \alpha_1, \ldots , \log \alpha_n$ are linearly independent over the rationals. 
Let $b_1, \ldots , b_n$ be integers with $b_n \not= 0$, and set $B = \max\{3, |b_1|, \ldots , |b_n|\}$. 
For $j =1, \ldots, n$, set $h^* (\alpha_j) = \max\{h(\alpha_j), 1\}$, where $h$ 
denotes the (logarithmic) Weil height. 
Assume that the quantity $\Lambda = b_1 \log \alpha_1 + \cdots + b_n \log \alpha_n$ is nonzero. 
A typical lower bound of $\log |\Lambda|$ given by Baker's theory of linear forms in logarithms takes the shape 
$$
\log |\Lambda| \ge - c(n, D)  \, h^* (\alpha_1) \cdots h^* (\alpha_n) \log B, 
$$
where $c(n,D)$ is positive, effectively computable and depends only on $n$ and on the degree $D$ of the field generated 
by $\alpha_1, \ldots , \alpha_n$. 
However, in certain special cases and in particular when $|b_n| = 1$, this bound can be improved to
$$
\log |\Lambda| - c(n, D)  \, h^* (\alpha_1) \cdots h^* (\alpha_n) \log \frac{B}{h^* (\alpha_n)}.
$$
The term $B / h^* (\alpha_n)$ in place of $B$ 
originates in works of Feldman and Baker 
and is a key tool for improving, in an effective way, the upper bound for the irrationality exponent
of a real algebraic number of degree at least $3$ given by 
Liouville's theorem.
We survey various applications of this refinement to exponents of approximation evaluated at algebraic numbers, 
to the $S$-part of some integer sequences, and to Diophantine equations. 
We conclude with some new results on arithmetical properties of convergents to real numbers. 
\end{abstract}


\maketitle

\section{Introduction}\label{sec:1}

Baker's theory of linear forms in the logarithms of algebraic numbers provides us with non-trivial, fully 
explicit lower bounds for the distance between $1$ and a product 
$\alpha_1^{b_1} \ldots \alpha_n^{b_n}$ of $n$ integer powers of algebraic numbers, that is (since 
$\log (1 + x)$ is equivalent to $x$ in a neighborhood of $0$), for the absolute value of the linear form
$$
b_1 \log \alpha_1 + \cdots + b_n \log \alpha_n. 
$$
%in terms of $n$, the heights of the algebraic numbers involved, their exponents, and the degree of the number field 
%the generated over $\Q$. 
It has been developed by Baker \cite{Ba66} and his followers since 1966 and has applications to 
Diophantine equations and many other topics; the interested reader is directed to the monographs 
\cite{Ba75,BaWu07,Bu18b,EvGy15,EvGy16,ShTi86,Spr93,WaLiv} and to the references quoted therein. 

In short, Baker's theory says that if the linear form in logarithms of algebraic numbers
$$
\Lambda := b_1 \log \alpha_1 + \cdots + b_n \log \alpha_n
$$ 
is nonzero, then its absolute value is at least equal to some effectively computable positive quantity, expressed 
in terms of $n$, the maximum $B$ of the absolute values of $b_1, \ldots , b_n$, and 
%the modified heights $h^* (\alpha_1), \ldots , h^* (\alpha_n)$ of 
the algebraic numbers $\alpha_1, \ldots , \alpha_n$. 

Throughout this paper, $h(\alpha)$ denotes the (logarithmic) Weil height of the 
algebraic number $\alpha$ and we set $h^* (\alpha) = \max\{h(\alpha), 1\}$. Recall that if $r/s$ is a nonzero 
rational number written in its lower form, then its height $h(r/s)$ is equal to the maximum of $\log r$ and $\log s$. 
A typical lower bound for $\log |\Lambda|$, established in \cite{BaWu93,Wa93}, takes the form
\begin{equation} \label{typ}
\log |\Lambda| \ge - c \, h^* (\alpha_1) \ldots h^* (\alpha_n) \log B,
\end{equation}
where $c$ is some effectively computable positive real number depending only on $n$ and on the degree of the 
number field generated by $\alpha_1, \ldots , \alpha_n$. But between the birth of the theory and the proof of such a 
clean result, there have been several steps and many difficulties to overcome. 
% We will briefly mention some of them, with a special focus on the dependence on $B$. 

First, we point out that a Liouville-type estimate yields a bound with $B$ instead of $\log B$ in \eqref{typ},
% (and the sum of the $\log A_i$ instead of their product), 
while, for most (not all) of the applications, a lower bound with a dependence on $B$ in $o(B)$ would be sufficient. 
The first result of Baker \cite{Ba66} involves a factor $(\log B)^{n+1+ \eps}$,  
where $\eps$ is an arbitrary positive real number.

We gather in Theorem \ref{lflog} below immediate consequences of estimates of 
Waldschmidt \cite{Wa93,WaLiv} and Matveev \cite{Mat00}; see also \cite[Theorems 1.1 and 1.2]{Bu18b}.

\begin{theorem}   \label{lflog} 
Let $n \ge 1$ be an integer. 
Let $\alpha_1, \ldots, \alpha_n$ be non-zero algebraic numbers. 
%Let $\log \alpha_1, \ldots , \log \alpha_n$ be determinations of their logarithm and 
%assume that $\log \alpha_1, \ldots , \log \alpha_n$ are linearly independent over the rationals. 
Let $b_1, \ldots , b_n$ be integers with $b_n \not= 0$.
Let $D$ be the degree over $\Q$ 
of the number field $\Q(\alpha_1, \ldots, \alpha_n)$. 
%Let $A_1, \ldots, A_n$ be real numbers with
%$$
%h^* (\alpha_j) \ge \max \{h(\alpha_j), 2\}, 
%\Bigl\{ h(\alpha_j), { \rme \over D} |\log \alpha_j|, {1 \over D} \Bigr\},
%\quad 1\le j \le n.
%$$ 
Set
$$
B = \max\{3, |b_1|, \ldots , |b_n|\},
$$ 
$$
B' = \max\Bigl\{3,   \max_{1 \le j \le n-1} 
\Bigl\{ {|b_n| \over h^* (\alpha_j)} + {|b_j| \over h^* (\alpha_n)} \Bigr\} \Bigr\}, 
$$
and
$$
B''= \max\Bigl\{3, \max\Bigl\{ |b_j| \ {h^* (\alpha_j) \over h^* (\alpha_n)} : 1 \le j \le n \Bigr\} \Bigr\}.
$$
Assume that 
$$
\Lambda := b_1 \log \alpha_1 + \cdots + b_n \log \alpha_n
$$ 
is nonzero. Then, there exist 
effectively computable positive numbers $c_1, \ldots , c_5$, depending only on $D$, 
such that the following holds. We have
\begin{equation} \label{eq1}
\log |\Lambda|    
\ge - c_1^n  \, h^* (\alpha_1) \ldots h^* (\alpha_n) \log B, 
\end{equation}
and, furthermore,  
\begin{equation} \label{eq2}
\log |\Lambda|    
\ge - c_2^n \, n^{3n} \, h^* (\alpha_1) \ldots h^* (\alpha_n) \log B',
\end{equation}
and
\begin{equation} \label{eq3}
\log |\Lambda|    
\ge - c_3^n  \, h^* (\alpha_1) \ldots h^* (\alpha_n) \log B''. 
\end{equation}
In particular, if $|b_n|=1$, then we have 
\begin{equation} \label{eq4}
\log |\Lambda|    
\ge - c_4^n \, n^{3n} \, h^* (\alpha_1) \ldots h^* (\alpha_n) \log \frac{B}{h^* (\alpha_n)} 
\end{equation}
and
\begin{equation} \label{eq5}
\log |\Lambda|    
\ge - c_5^n  \, h^* (\alpha_1) \ldots h^* (\alpha_n) \log \frac{B \max\{h^* (\alpha_1), \ldots , h^* (\alpha_{n-1}) \}}{h^* (\alpha_n)}.
\end{equation}
\end{theorem} 

The estimate \eqref{eq1} takes the shape $|\Lambda| \ge B^{-C}$, where 
$C$ depends only on $n, \alpha_1, \ldots , \alpha_n$. However, it does not 
include all the refirements. In particular, it is weaker than \eqref{eq3} in some important cases.

The quantity $B'$ in \eqref{eq2} originates in Feldman's papers \cite{Fe68,Fe71}. It is 
a consequence of the use of the functions $x \mapsto {x \choose k}$ instead of 
$x \mapsto x^k$ in the construction of the auxiliary function. 
% and is the key ingredient 
%for Feldman's effective improvement of Liouville's inequality. 
The key point is the presence of the factor $h^* (\alpha_n)$ in the denominator in the
definition of $B'$. It is of great interest when $|b_n| = 1$ and $h^* (\alpha_n)$ is
large, since it then allows us, roughly 
speaking, to replace $B$ by the smaller quantity $B / h^* (\alpha_n)$; see \eqref{eq4}. 
The difference between \eqref{eq2} and \eqref{eq3} (and between \eqref{eq4} and \eqref{eq5}) lies mainly in the 
dependence on $n$. The improvement in the dependence on $n$ is Matveev's breakthrough  \cite{Mat00}. 
Observe that $B''$ may be larger than $B'$. This, however, does not cause any trouble in most 
of the applications (but not in all of them). 

The term $B'$ has also its origin in the {\it Sharpenings II} of Baker \cite{Ba73}, where the 
following statement is established (with a different notion of height). 
Throughout this paper, we use $\gg_{a, b, \ldots}$ and $\ll_{a, b, \ldots}$ 
to indicate that the positive numerical constants implied by $\gg$ and $\ll$ are effectively computable and depend 
at most on the parameters $a, b, \ldots$. 

\begin{theorem} \label{Baker73}
Keep the notation of Theorem \ref{lflog}. 
Set $B_n = \max\{3, |b_n|\}$ and $B_0 = \max\{ 3, |b_1|, \ldots , |b_{n-1}|\}$. 
There exists an effectively computable number $C$, depending only on $n, D, \alpha_1, \ldots , \alpha_{n-1}$ such that, 
%Let the heights of $\alpha_1, \ldots , \alpha_n$
%and $\alpha_{n+1}$ be at most $A$ and $A_{n+1}$ respectively, 
%where $A \ge 2$ and $A_{n+1} \ge 2$. 
for any real number $\delta$ with $0 < \delta < 1/2$, we have 
%Let $b_1, \ldots , b_n$ be integers, not all zero, and set $B = \max\{ 3, |b_1|, \ldots , |b_n|\}$. 
$$
|\Lambda| \ge (\delta / B_n)^{C h^* (\alpha_n)} \, \rme^{- \delta B_0}. 
$$
%\end{theorem}
%We highlight the special case $b_n = \pm 1$. 
%\begin{theorem} \label{Baker73b}
%Keep the notation of Theorem \ref{lflog}. 
In particular, if we assume that $n \ge 2$ and $b_n = 1$, and 
%Let the heights of $\alpha_1, \ldots , \alpha_n$
%and $\alpha_{n+1}$ be at most $A$ and $A_{n+1}$ respectively, 
%where $A \ge 2$ and $A_{n+1} \ge 2$. 
if $\eta$ is a real number with $0 < \eta < 1$, then
%Let $b_1, \ldots , b_n$ be integers, not all zero, and set $B = \max\{ 3, |b_1|, \ldots , |b_n|\}$. 
the inequality 
$$
|\Lambda|  < \rme^{- \eta B}, 
$$
implies the upper bound
$$
B = B_0 \ll_{n, D, \alpha_1, \ldots , \alpha_{n-1}, \eta} \, h^* (\alpha_n). 
$$
\end{theorem}

The last statement of Theorem \ref{Baker73} follows from the first one by taking $\delta = \eta / 2$. 

Theorem \ref{Baker73} improved the earlier results to the extent of the elimination of 
a factor $\log h^* (\alpha_n)$ from the bound for $B_0$. 
Note, however, that the dependence on $h^* (\alpha_1), \ldots , h^* (\alpha_{n-1})$ is not specified.

Feldman \cite[Theorem 1]{Fe71} obtained a similar result two years before, under some restrictions on 
$\alpha_n$, which enabled him to get the first effective improvement on Liouville's bound for 
the irrationality exponent of an algebraic number of degree at least $3$, see Theorem \ref{feld} below
(and \cite[p. 118]{Ba72}).

The second assertion of Theorem \ref{Baker73} is implied by \eqref{eq4}. Indeed, under the assumption 
$\log |\Lambda| \le - \delta B$, we get
$$
B \le \delta^{-1} \,  c_4^n \, n^{3n} \, h^* (\alpha_1) \ldots h^* (\alpha_n) \log \frac{B}{h^* (\alpha_n)},
$$
hence an upper bound for $B$ of the form 
$$
B \ll_{n, D, \alpha_1, \ldots , \alpha_{n-1}, \delta} \, h^* (\alpha_n). 
$$
See also \cite{BiBu00} for an alternative proof. 

Theorem \ref{Baker73} is sufficient in most of the situations, but not in all of them. 
Assume for instance that the upper bound for $|\Lambda|$ takes the form 
$\log |\Lambda| \le - \delta B h^* (\alpha_1)$. Combined with \eqref{eq4}, this gives an upper bound for $B$
which is linear in $h^* (\alpha_n)$ and independent of $h^*(\alpha_1)$, 
while Theorem \ref{Baker73} yields an upper bound linear in $h^* (\alpha_n)$, but also 
dependent on $h^* (\alpha_1)$. 
The same example illustrates the difference between \eqref{eq4} and \eqref{eq5}. 
As we will see in Subsection \ref{appsum}, 
we encouter such situations when $\alpha_1$ (or one number among $\alpha_1, \ldots , \alpha_{n-1}$) is unknown. 
Additional explanations are given in \cite[p. 361-362]{WaLiv}.

The term $B'$ has apparently been first introduced in \cite{LMPW87}, but with a slightly different definition. 
It is called $M$ in \cite{Wa93} and $b'$ in \cite{Lau94}. 

There is a non-Archimedean analogue of Baker's classical theory. 
Let $p$ be a prime number. 
For a nonzero rational number $\alpha$, let $\vv_p (\alpha)$ denote the exponent of $p$ in the decomposition of $\alpha$ 
as a product of powers of prime numbers. 
More generally, if $\alpha$ is a nonzero algebraic number in a number field $K$, 
let $\vv_{\pp} (\alpha)$ denote the
exponent of $\pp$ in the decomposition of the
fractional ideal $\alpha O_K$ in a product of prime ideals and set 
$$
\vv_p (\alpha) = {\vv_{\pp} (\alpha) \over e_{\pp}}.
$$
This defines a valuation $\vv_p$ on $K$ which extends 
the $p$-adic valuation $\vv_p$ on $\Q$ normalized in such a way that $\vv_p (p) = 1$. 
We reproduce, in a simplified form, estimates obtained by Yu \cite{Yu07}; see also \cite[Theorems 2.9 and 2.11]{Bu18b}. 

\begin{theorem} \label{Yu} 
Let $n \ge 2$ be an integer. 
Let $p$ be a prime number and $\alpha_1, \ldots, \alpha_n$ 
algebraic numbers in an algebraic number field of degree $D$.
Let $b_1, \ldots , b_n$ denote rational integers such that $\alpha_1^{b_1} \ldots \alpha_n^{b_n}$
is not equal to $1$. 
%Let $A_1, \ldots , A_n, B$ be real numbers with
%$$
%h^* (\alpha_j) \ge  \max\{ h(\alpha_j), 2\}, 
%\max\Bigl\{ h(\alpha_j), {1 \over 16 \rme^2 D^2} \Bigr\},
%\quad 1\le j \le n,   
%$$
%and
Set
$$
B  =  \max\{3, |b_1|, \ldots , |b_n| \}.   
$$
There exist positive effectively computable real numbers $c_1,  \ldots , c_7$,  depending only on $D$, such that the 
following holds. 
%If $n \ge 2$, 
We have
$$
\vv_p (\alpha_1^{b_1} \ldots \alpha_n^{b_n} - 1) < 
c_1^n \, {p^D} \, h^* (\alpha_1) \ldots h^* (\alpha_n) \log B.
$$
%where $c_6$ is an effectively computable number depending only on $D$. 
Let $B_n$ be a real number such that
$$
B \ge B_n \ge |b_n|.
$$
Assume that
$$
\vv_p (b_n) \le \vv_p (b_j), \quad j = 1, \ldots , n.
$$
Let $\delta$ be a real number with $0< \delta \le {1 \over 2}$. 
Then, we have
\begin{equation} \label{eqYu}
\vv_p (\alpha_1^{b_1} \ldots \alpha_n^{b_n} - 1)  < c_2^n \, {p^D \over (\log p)^2}   \, 
\max\Bigl\{  h^* (\alpha_1) \cdots h^* (\alpha_n) (\log T), {\delta B \over B_n c_3^n } \Bigr\},
%\eqalign{
%\vv_p (\alpha_1^{b_1} \ldots \alpha_n^{b_n} - 1)  < \, &
%c_7^n \, p^D \times \cr
% {p^D - 1 \over (\log p)^2} \times \cr
%&  \, \, \, \, 
%\times \max\Bigl\{  (\log A_1) \cdots (\log A_n) (\log T), {\delta B \over B_n c_8^n } \Bigr\}, \cr}
\end{equation}
where
$$
T = {B_n \over \delta} c_4^{n^2} p^{(n+1)D}  h^* (\alpha_1) \cdots h^* (\alpha_{n-1}). 
$$
%and $c_1$ is an effectively computable number depending only on $D$.  
In particular, 
%Then, 
we get either
\begin{equation} \label{eq6}
B \le  c_5^n \, h^* (\alpha_1) \cdots h^* (\alpha_n) \, B_n 
\end{equation}
or
%$$
%\eqalign{
%\vv_p (\alpha_1^{b_1} \ldots \alpha_n^{b_n} - 1)  < &  
%c_5^n \, {p^D \over (\log p)^2}  \times \cr
%& \, \, \, \,  \times (\log A_1) \cdots (\log A_n) \, \max\Bigl\{1, 
%\log \Bigl( {c_7^{n^2}  p^{(n+1)D} B \over c_6^n \log A_n} \Bigr) \Bigr\}, \cr}     \eqno (2.10) 
%$$
%that is,
\begin{equation} \label{eq7}
\vv_p (\alpha_1^{b_1} \ldots \alpha_n^{b_n} - 1)  < c_6^n \, p^D \, 
%{p^D \over \log p} \, 
h^* (\alpha_1) \cdots h^* (\alpha_n) \, \log \max \Bigl\{3, \frac{B}{h^* (\alpha_n)} \Bigr\}.  
\end{equation}
Furthermore, we have
\begin{equation} \label{eq8}
\begin{split}
\vv_p (\alpha_1^{b_1} \ldots \alpha_n^{b_n} - 1)  < c_7^n \, & p^D \, 
%{p^D \over \log p} \, 
h^* (\alpha_1) \cdots  h^* (\alpha_n) \\
%&  \log \max \Bigl\{\frac{B}{h^* (\alpha_n)},  \max \{h^* (\alpha_1), \ldots , h^* (\alpha_{n-1}) \}   \Bigr\}.  
&  \log \frac{B \max \{h^* (\alpha_1), \ldots , h^* (\alpha_{n-1}) \} }{h^* (\alpha_n)}.
%(\log A_1) \cdots (\log A_{n-1}) \Bigr\}.  
\end{split}
\end{equation}
%where $c_7$ and $c_8$ are effectively computable numbers depending only on $D$. 
\end{theorem} 

The derivation of `\eqref{eq6} or \eqref{eq7}' (\resp of \eqref{eq8}) from \eqref{eqYu} is explained 
on \cite[p. 20]{Bu18b} (\resp on \cite[p. 57]{EvGy15}).

Clearly, \eqref{eq7} is a $p$-adic analogue of \eqref{eq4} and \eqref{eq5}. 
But, this is not the conclusion of Theorem \ref{Yu}: there is an alternative, whose other member 
is \eqref{eq6}. In most of the cases, this is sufficient for our applications, but there are examples 
for which this is not, see Subsection \ref{appsum}. 

In the case of two logarithms, an analogue of \eqref{eq4} has been proved in \cite{BuLa96}, with, however, 
a weaker dependence on $B'$. 

\begin{theorem}    \label{BuLa}
Let $p$ be a prime number. 
Let $\alpha_1$ and $\alpha_2$ be multiplicatively independent algebraic numbers 
with $\vv_p (\alpha_1) = \vv_p (\alpha_2) = 0$.
Set $D = [\Q(\alpha_1, \alpha_2) : \Q]$.
%Let $A_1$ and $A_2$ be real numbers with 
%$$
%\log A_j \ge  \max\Bigl\{\h(\alpha_j), {\log p \over D}\Bigr\},
%h^* (\alpha_j) \ge \max\{h(\alpha_j), 2\},  \quad j =1, 2. 
%$$
Let $b_1$ and $b_2$ be positive integers and set
$$
B' = \max\Bigl\{ 3,  {b_1 \over h^* (\alpha_2)} + {b_2 \over h^* (\alpha_1)}  \Bigr\}. 
%\log B' = \max\Bigl\{ \log \Bigl(  {b_1 \over D \log A_2} + {b_2 \over D \log A_1} \Bigr)  
%+ \log\log p + 0.4, {10\log p\over D},10 \Bigr\}. 
$$
Then, there exists an absolute, 
effectively computable number $c_1$ such that
$$
\vv_p (\alpha_1^{b_1} - \alpha_2^{b_2} ) 
\le  c_1 p^D \,
h^* (\alpha_1) h^* (\alpha_2)   \, (\log B')^2.   
$$
\end{theorem} 

We have slightly simplified the statement of Theorem \ref{BuLa}. Its interest is more practical than theoretical: indeed, the 
numerical constant (which is not reproduced here) is much smaller than those in Theorem \ref{Yu}, and this 
is crucial for applications to the resolution of 
Diophantine equations, even in spite of the appearance of $(\log B')^2$. 
In principle, the square over the factor $(\log B')$ can be removed. This has been worked out 
in the Archimedean case by Gouillon \cite{Gou06}. 

Also, in principle, it should be possible to replace the conclusion `we get \eqref{eq6} or \eqref{eq7}' by 
a weaker form of \eqref{eq7}, with an extra factor $n^{2n}$ or $n^{3n}$. 
This is an open problem, which would have some applications (see Subsection \ref{appsum}), even if the 
dependence on $n$ is not very good. 

\begin{problem}
Keep the assumption of Theorem \ref{Yu}. Assume that $|b_n| = 1$. 
Establish that there exists a positive effectively computable real number $c_1$,  depending only on $D$, such that
$$
\vv_p (\alpha_1^{b_1} \ldots \alpha_n^{b_n} - 1)  < c_1^n \, n^{3n} \, p^D \, 
h^* (\alpha_1) \cdots h^* (\alpha_n) \, \log \max \Bigl\{3, \frac{B}{h^* (\alpha_n)} \Bigr\}.  
$$
\end{problem}

%The purpose of this paper is to survey many different situations, where there is a big gap 
%between the effective estimate (trivial or coming from a Liouville type argument) 
%and the ineffective estimate (obtained by using Roth's theorem or the Schmidt Subspace Theorem), 
%and where an application of the theory of linear forms in logarithms, where $B'$ plays a key r\^ole, 
%yields a slight, but non-trivial, improvement on the effective estimate. 

The purpose of this paper is to survey several questions, where the presence of $B'$ yields stronger, 
or more general, results than the usual estimate \eqref{eq1}. We discuss five exponents of Diophantine 
approximation in Section \ref{sec:2}. Sections \ref{sec:3} to \ref{sec:5} are devoted to Diophantine equations.
Section \ref{sec:6} contains some new results on arithmetical properties of convergents to real numbers. 

This paper can be seen as a complement to the monograph \cite{Bu18b}, even if several results 
surveyed below were already discussed in \cite{Bu18b}.

\section{Exponents of Diophantine approximation}  \label{sec:2}

In this section, we discuss the values taken by several exponents of Diophantine approximation 
at algebraic arguments. 
In each of the five examples considered below, the replacement of $B$ by $B'$ is crucial to get
a (small) saving on the easy bound derived from 
a Liouville-type argument. In Sections \ref{ssec1} to \ref{ssec4} the proofs need some (classical and elementary) algebraic number theory, 
and we do not reproduce them here. The interested reader is directed to \cite{Bu18b} and/or to the original papers.

\subsection{Irrationality exponent}   \label{ssec1}

The irrationality exponent  
of an irrational number $\theta$
measures the quality of approximation to $\theta$ by rational numbers. 

\begin{definition}     \label{IrrExp}  
Let $\theta$ be an irrational real number. 
We denote by $\mu (\theta)$ (resp.,  $\mueff (\theta)$) the infimum of the real numbers $\mu$ 
for which there exists 
%\indexbf{irrationality measure} 
a positive real number $C(\theta)$ (resp.,  a positive, effectively computable, real number $C(\theta)$)
such that every rational number ${p \over q}$
with $q \ge 1$ satisfies 
$$
\Bigl| \theta - {p \over q} \Bigr| > {C(\theta) \over q^{\mu}}.
$$
\end{definition}

Clearly, $\mueff (\theta)$ is always larger than or equal to $\mu (\theta)$, which is always at least equal to $2$, by 
the theory of continued fractions. 

Let $\xi$ be a real, algebraic, irrational number of degree $d$. Liouville's theorem asserts that there exists a 
positive, effectively computable number $C(\xi)$ such that 
\begin{equation}  \label{Liouv}
\Bigl| \xi - {p \over q} \Bigr| > {C(\xi) \over q^d}, \quad \hbox{for every $p, q$ in $\Z$ with $q \ge 1$.} 
\end{equation} 
Thus, we have $\mueff(\xi) \le d$, while $\mu(\xi) = 2$, by 
Roth's theorem. For $d \ge 3$, Liouville's bound for $\mueff(\xi)$ can be slightly improved.  
This was first established by Feldman \cite{Fe71}. 

\begin{theorem}    \label{feld}
Let $\xi$ be a real, algebraic number of degree $d$ at least $3$. 
There exists a positive, effectively computable $\tau = \tau (\xi)$ such that $\mueff(\xi) \le d - \tau$.
\end{theorem}

The term $B'$ is at the source of the proof that $\mueff(\xi) < d$. 
By using the lower bound \eqref{eq1} in place of \eqref{eq2}, we get an estimate of the form
$$
\Bigl| \xi - \frac{p}{q} \Bigr| >  \frac{1}{q^{d - c / \log \log q}},
$$ 
for some effectively computable, positive $c$ and every $q \ge 10$. This effective improvement of Liouville's inequality \eqref{Liouv} 
does not imply anything better than $\mueff(\xi) \le d$. 
%[To be checked].

%\begin{itemize}
%\item[(i)] $\mueff(\xi) \le d$ (Liouville's bound); 
%\item[(ii)] $\mu(\xi) = 2$ (Roth's theorem, ineffective);
%\item[(iii)] There exists $\tau(\xi) > 0$ such that $\mueff(\xi) \le d - \tau(\xi)$ (Feldman's bound);
%\item[(iv)] For some specific $\xi$, we have effective improvements of Feldman's bound.
%\end{itemize}

%The key tool for Feldman's improvement of Liouville's upper bound is his (apparently, slight) 
%strengthening of an estimate of Baker for linear forms in logarithms. 

In some special cases, for instance for $\xi = \root{3} \of{2}$, Theorem \ref{feld} can be considerably improved; see 
\cite{Ba64} and subsequent works. 
Furthermore, Bombieri and Mueller \cite{BoMu83} and 
Bombieri, van der Poorten, and Vaaler \cite{BovdPVa96} proved that, for every $\eps > 0$ and every integer $d \ge 3$, there exist
algebraic real numbers $\xi$ of degree $d$ such that $\mueff (\xi) < 2 + \eps$. 
In \cite{BoMu83} the authors consider integral roots of rational numbers close to $1$, while in \cite{BovdPVa96} the 
algebraic numbers are roots of cubic polynomials $x^3 + p x + q$, where $p > 0$ and $q$ are relatively prime integers, with 
$p$ sufficiently large in terms of $|q|$.

\subsection{On the $b$-ary expansion of an algebraic number}   \label{ssec2}

Let $\theta$ be an irrational real number. 
We denote by $v_b (\theta)$ the infimum
of the real numbers $v$ for which the inequality
$$
\| b^n \theta \| > (b^n)^{-v}
$$
holds for every sufficiently large positive integer $n$.
Likewise, $\vbeff (\theta)$ denotes the infimum of the real numbers $v$
for which there exists an effectively computable
integer $n_0 (v)$ such that
$$
\| b^n \theta \| >   (b^n)^{-v},
$$
for $n \ge n_0(v)$. 

Assume that $\xi$ is algebraic of degree $d \ge 2$. 
Then Liouville's inequality yields the trivial upper bound 
$$
\vbeff (\xi) \le \mueff(\xi) - 1 \le d-1, 
$$ 
while Ridout's theorem \cite{Rid57} implies that $v_b (\xi) = 0$. 
Furthermore, if $d \ge 3$, then it follows from Theorem \ref{feld} that 
\begin{equation} \label{r7}
\vbeff (\xi) \le d-1-\tau,   
\end{equation}
for some effectively computable positive real number $\tau = \tau(\xi)$.    
For $d = 2$, the theory of $p$-adic linear forms in logarithms allows us 
to improve the trivial estimate $\vbeff (\xi) \le 1$ as follows.

\begin{theorem} \label{th:vb}
For every integer $b \ge 2$ and every 
quadratic real number $\xi$, we have
$$
\vbeff (\xi) \le 1 - \tau, 
$$
%for $n \geq 1$,
%where $c(\xi, {\mathcal P})$ and $\tau(\xi, {\mathcal P})$ are positive computable constants 
%depending only on $\xi$ and on the set ${\mathcal P}$ of prime factors of $b$.
where $\tau = \tau(\xi, S)$ is a positive, effectively computable, constant
depending only on $\xi$ and on the set $S$ 
of prime factors of $b$. 
\end{theorem}

In some specific cases, we can get better upper bounds and even, for every $\eps > 0$, construct positive integers $m$ such that 
$\vbeff (\sqrt{m} ) < \eps$. 
Namely, it has been proved in \cite{BeBu12} that for every integers $b \ge 2$ and $k \ge 1$ we have
$$
\vbeff (\sqrt{b^{2k} + 1} ) \le \frac{\log 48}{k \log b}.
$$

\subsection{Simultaneous Pell equations}   \label{ssec3}

The definition of (effective) irrationality exponent extends as follows. 
%\proclaim Definition 1.1. 
Let $\theta, \theta'$ be real numbers such that $1, \theta, \theta'$ are linearly independent 
over the rational numbers. We denote by $\mu (\theta, \theta')$ (\resp  $\mueff (\theta, \theta')$) 
the supremum of the real numbers $\mu$ 
%is a simultaneous irrationality measure for the pair $(\xi, \zeta)$ if 
for which there exists 
a positive real number $c(\theta, \theta')$ 
(\resp a positive, effectively computable, real number $c(\theta, \theta')$) 
such that, for every integer triple $(p, q, r)$
with $q \ge 1$, we have 
$$
\max \Bigl\{ \Bigl| \theta - {p \over q} \Bigr|,  \Bigl| \theta' - {r \over q} \Bigr|  \Bigr\} > 
{c(\theta, \theta') \over q^{\mu}}.
$$

Let $\theta, \theta'$ be real numbers such that $1, \theta, \theta'$ are linearly independent 
over the rational numbers. An easy application of Minkowski's theorem implies 
that $\mu (\theta, \theta') \ge {3 \over 2}$ and a covering lemma shows that equality holds 
for almost all pairs $(\theta, \theta')$, with respect to the planar Lebesgue measure. 
Schmidt \cite{Schm67} established 
%, first step towards 
%his famous Subspace Theorem \cite{Schm70}, 
that $\mu (\xi, \zeta) = {3 \over 2}$ if $\xi$ and $\zeta$ are both real and algebraic. 
His result is ineffective and gives no better information on 
$\mueff (\xi, \zeta)$ than the obvious inequality 
$$
\mueff (\xi, \zeta) \le \max\{ \mueff(\xi), \mueff(\zeta) \}.
$$
The particular case where $\xi$ and $\zeta$ are quadratic numbers in distinct 
number fields is of special interest and was considered in \cite{Bu20}.

\begin{theorem}
Let $\xi, \zeta$ be real quadratic numbers in distinct quadratic fields. 
%such that $1, \xi, \zeta$ are linearly independent over the rational numbers. 
%Let $R_\xi$ and $R_\zeta$ denote the regulators
%of the fields $\Q(\xi)$ and $\Q(\zeta)$, respectively. 
Then, we have $\mu (\xi, \zeta) = 3/2$ and
there exists a positive, effectively computable real number $\tau$, depending only on $\xi$ and $\zeta$, such that 
$$
\mueff (\xi, \zeta) \le 2 - \tau. 
%\mueff (\xi, \zeta) \le 2 - (c R_\xi R_\zeta \log (R_\xi R_\zeta))^{-1}. 
$$
\end{theorem}

Better upper bounds 
%than the obvious $\mueff (\xi, \zeta) \le 2$ 
have been obtained in some special cases, 
in particular by Rickert \cite{Ri93} (see his paper for earlier references), who 
established among other results that 
$$
\mueff (\sqrt 2, \sqrt 3) \le 1.913, 
$$
and subsequently by Bennett \cite{Be95,Be96}. 
Their method applies only to a restricted 
class of pairs $(\xi, \zeta)$ of quadratic numbers. 
However, it is strong enough to establish that, for every $\eps > 0$ 
and every positive integer $N$ sufficiently large in terms of $\eps$, we have 
$$
\mueff \Bigl( \sqrt{ 1 - \frac{1}{N}},  \sqrt{ 1 + \frac{1}{N}} \, \Bigr) < \frac{3}{2} + \eps, 
$$
see \cite{Ri93}.

\subsection{Multiplicative $p$-adic approximation}   \label{ssec4}

Let $p$ be a prime number and $\theta$ be an irrational, $p$-adic number. 
We denote by $\tmu (\theta)$ the infimum
of the real numbers $\mu$ for which the inequality
\beq  \label{mudef}
| b \theta - a |_p  > |ab|^{-\mu}
\eeq 
holds for every nonzero integers $a, b$ with $|ab|$ sufficiently large. 
Likewise, $\tmueff (\theta)$ denotes the infimum of the real numbers $\mu$
for which there exists an effectively computable
integer $A(\mu)$ such that \eqref{mudef}
holds for every integers $a, b$ with $|a b| \ge A(\mu)$. 

Assume that $\alpha$ is a $p$-adic algebraic number of degree $d \ge 2$. 
Then Liouville's inequality yields the trivial effective lower bound 
$$
|b \alpha - a |_p \gg_\alpha  \max\{|a|, |b|\}^{-d}, \quad \hbox{for $a, b$ in $\Z_{\not= 0}$}, 
$$
thus, since $\max\{|a|, |b|\} \le |ab|$, we have
$$
\tmueff (\alpha) \le  \mueff(\alpha) \le d, 
$$ 
while Ridout's theorem \cite{Rid58} implies that $\tmu (\alpha) = 1$. 
Furthermore, if $d \ge 3$, then the analogue of Theorem \ref{feld} holds (see \cite[Section V.2]{Spr93}), namely  we have
\begin{equation} \label{feldpadic}
\tmueff (\alpha) \le d - \tau,   
\end{equation}
for some effectively computable positive real number $\tau = \tau(\alpha)$.    
For $d = 2$, the theory of Archimedean linear forms in logarithms allows us 
to improve the trivial estimate $\tmueff (\alpha) \le 2$ as follows; see \cite{Bu22b}. 

\begin{theorem} \label{th:tmu}
For every prime number $p$ and every 
$p$-adic quadratic number $\alpha$, 
there exists a positive, effectively computable real number $\tau = \tau (\alpha)$ such that 
$$
\tmueff (\alpha) \le 2 - \tau. 
$$
\end{theorem}

We conclude this subsection with an open question. 

\begin{problem} \label{Pbpmult}
For a given prime number $p$ and an arbitrary $\eps > 0$, to construct 
quadratic $p$-adic numbers $\alpha$ such that $\tmueff (\alpha) \le 1 + \eps$. 
\end{problem}

\subsection{Fractional parts of powers of real algebraic numbers}  \label{ssec5}

For a real number $x$, let 
$$
||x|| = \mbox{min}\{|x-m|:m\in\mathbb{Z}\}
$$
denote the distance to its nearest integer. 
In 1957 Mahler \cite{Mah57} applied Ridout's $p$-adic extension \cite{Rid57} 
 of Roth's theorem to prove the first assertion of the following result. 
The second assertion was proved in 2004 by Corvaja and Zannier \cite{CZ04}, 
who applied ingeniously the $p$-adic Schmidt Subspace Theorem. 
Recall that a Pisot number is a real algebraic integer greater than $1$ with the property 
that all of its Galois conjugates (except itself) lie in the open unit disc.

\begin{theorem}\label{MahCZ}
Let $r/s$ be a rational number greater than $1$ and which is not an integer. 
Let $\eps$ be a positive real number. 
Then, there exists an integer $n_0$ such that 
$$
\Big\| \Bigl( \, {r \over s} \, \Bigr)^n \Bigr\| > s^{- \eps n}, 
$$
for every integer $n$ exceeding $n_0$. 
More generally, let $\xi$ be a real algebraic number greater than $1$.  
If there are no positive integers $h$ such that $\xi^h$ is a Pisot number, then 
there exists an integer $n_0$ such that 
$$
\| \xi^n \| > \xi^{- \eps n}, 
$$
for every integer $n$ exceeding $n_0$.  
\end{theorem}

Theorem  \ref{MahCZ} is ineffective in the sense that its proof does not yield an 
explicit value for the integer $n_0$. 
To get an effective improvement on the trivial estimate $\| (r/s)^n \| \ge s^{-n}$, 
Baker and Coates \cite{BaCo75} (see also \cite{Bu02} and \cite[Section 6.2]{Bu18b})
applied the theory of linear forms in $p$-adic logarithms, with a prime number $p$ which divides~$s$.

\begin{theorem}\label{BC}
Let $r/s$ be a rational number greater than $1$ and which is not an integer. 
Then, there exist effectively computable positive real numbers $\tau = \tau(r/s)$ 
and $n_0 = n_0 (r/s)$ such that 
$$
\Big\| \Bigl( \, {r \over s} \, \Bigr)^n \Bigr\| > s^{- (1 - \tau) n}, 
$$
for every integer $n$ exceeding $n_0$. 
\end{theorem}

\begin{proof}
Let $n$ be a positive integer
and $A_n$ denote the nearest integer to $(r/s)^n$.
Set 
$$
m_n := r^n - A_n s^n.    
$$
Let $p$ be a prime divisor of $s$. 
By applying the first assertion of Theorem \ref{Yu} to estimate $\vv_p (r^n - m_n)$ we get that 
$$
n \le \vv_p (r^n - m_n) \ll_{r, s} \,  (\log |2 m_n| ) (\log n),
$$
hence
$$
n \ll_{r,s} \, (\log |2 m_n| ) \, (\log \log |2 m_n| ),
$$
which is not sufficient for our purpose. Fortunately, the second assertion of Theorem \ref{Yu} shows that
$$
n \ll_{r, s} \, \log |2 m_n| 
$$
or
$$
n \le \vv_p (r^n - m_n) \ll_{r, s} \,  (\log |2 m_n| ) \Bigl( \log {n \over  \log |2 m_n|} \Bigr),
$$
giving in both cases that 
\beq  \label{bound}
n \ll_{r, s} \log |2 m_n|.
\eeq 
Alternatively, Theorem \ref{BuLa} implies that 
$$
n \le \vv_p (r^n - m_n) \ll_{r, s} \,  (\log |2 m_n| ) \Bigl( \log {n \over  \log |2 m_n|} \Bigr)^2,
$$
and we get \eqref{bound} as well. 
This implies the existence of a positive real number $\delta$, depending 
only on $r$ and $s$, such that
$$
2 |m_n| \ge s^{\delta n}.
$$
By the definition of $m_n$ we conclude that 
$$
\Big\| \Bigl( \, {r \over s} \, \Bigr)^n \Bigr\| = {|m_n| \over s^{n}} \ge {1 \over 2 s^{(1 - \delta) n}}, \quad n \ge 1. 
$$
This completes the proof of the theorem. 
\end{proof}

For an arbitrary irrational, real algebraic number $\xi > 1$, 
Liouville's theorem also allows us, when $\xi^n$ is not an integer, to bound $\| \xi^n \|$ from below by a 
positive number raised to the power $-n$.  
This bound can be improved in an effective way when $\xi^n$ is neither an integer, nor a quadratic 
Pisot unit. 

\begin{theorem}\label{Main}
Let $\xi > 1$ be a real algebraic 
number of degree $d \ge 1$. 
Let $a_d$ denote the leading 
coefficient of its minimal defining polynomial over $\Z$ and 
$\xi_1, \ldots , \xi_d$ its Galois conjugates, ordered in such a way 
that $|\xi_1| \le \ldots \le |\xi_d|$. Let $j$ be such that $\xi = \xi_j$. 
Set 
$$
C_\xi = a_d  \, \xi^{d-1} \, \prod_{i > j} {|\xi_i| \over  \xi }.     
$$
If $\xi$ is not the $d$-th root of an integer, then we have 
%an integer root of an integer, then we have
\beq \label{nuLiouv}
\| \xi^n\| \ge 3^{-(d-1)} \, C_\xi^{-n}, \quad \hbox{for $n \ge 1$}.    
\eeq 
Otherwise, \eqref{nuLiouv} holds only for the positive integers $n$ such that 
$\xi^n$ is not an integer. 
%Let $h$ be the smallest positive integer such that $\xi^h$ is an integer or a quadratic Pisot unit
%and put $\cN_\xi = \{ h n : n \in \Z_{\ge 1} \}$. 
%If no such integer exists, then $\cN_\xi$ is the empty set.  
%If there exist no positive integers $h$ such that $\xi^h$ is a quadratic Pisot unit, 
%There exist a positive, effectively computable real number $\tau = \tau(\xi)$ 
%and an effectively computable integer $n_0 = n_0(\xi)$, both 
%depending only on $\xi$,  such that 
%$$
%\| \xi^n\|  \ge C_\xi^{- (1 - \tau) n}, \quad \hbox{for $n > n_0$ not in $\cN_\xi$.} 
%$$
\end{theorem}

For a real number $\theta$ which is not an integer %, nor an integer root of an integer, 
define 
$$
\nu (\theta) = \limsup_{n \to + \infty,  \theta^n \notin \Z} \, {- \log \| \theta^n \| \over n}
$$
and let $\nueff (\theta)$ denote the infimum of the real numbers $\nu$ 
for which there exists an effectively computable integer $n_0 = n_0 (\theta)$ such that 
$(- \log \| \theta^n \| )/n \le \nu$ for $n \ge n_0$ and $\theta^n$ is not an integer. 

The first two assertions of the next theorem are restatements of 
Theorems \ref{MahCZ} and \ref{Main}, while the last one was established in \cite{Bu22}. 

\begin{theorem}    \label{Bcras} 
Let $\xi > 1$ be an algebraic real number which is not an integer. 
Then, $\nu (\xi) = 0$, unless $\xi$ is an integer root of a Pisot number.  
Furthermore, $\nueff (\xi) \le \log C_\xi$ and, if 
$\xi$ is not an integer root 
%of an integer or 
of a quadratic Pisot unit, then 
there exists a positive, effectively computable real number $\tau = \tau(\xi)$ 
such that $\nueff (\xi) \le (1 - \tau) \log C_\xi$. 
\end{theorem} 

Sometimes, the hypergeometric method yields effective improvements of Theorem \ref{Bcras}.
This is the case for the algebraic numbers $\sqrt{2}$ and $3/2$, see 
Beukers' seminal papers \cite{Beu80,Beu81} 
and the subsequent works \cite{BaBe02,Zu07} where it is shown that 
$$
\nueff ( \sqrt{2} ) \le 0.595, \quad \nueff ( 3/2) < 0.5443, 
$$
respectively. 

Furthermore, for every $\eps > 0$, Bennett \cite{Be93a} constructed an infinite set $\cS_\eps$ of rational numbers,
dense in $(1, + \infty)$, such that every $p/q$ in $\cS_\eps$ satisfies $\nueff(p/q) < \eps \log q$, that is, 
$\nueff(p/q) < \eps \log C_{p/q}$; see also \cite{Be93b}.

\subsection{Summary}  \label{ssec6}

We have seen that, in each of the five subsections above (with one exception, see Problem \ref{Pbpmult}), we have:

\begin{itemize}
\item[(i)] an easy, effective bound $M$ (obvious or coming from a Liouville-type inequality); 
\item[(ii)] the exact result $m$ (but with an ineffective proof);
\item[(iii)] a small effective improvement $M - \tau$ on the easy bound (by using the $B'$); 
\item[(iv)] a more substantial effective improvement in certain specific cases, with even, for every $\eps > 0$, 
some explicit examples for which one gets an effective bound smaller than $m + \eps$. 
\end{itemize}

Moreover, in the cases where the effective bound is very close to the exact value obtained by ineffective methods, the 
specific numbers constructed with this property 
are in most cases very close to $1$, for the usual or for a $p$-adic absolute value.

\section{On the $S$-part of integer sequences}   \label{sec:3}

Throughout this section, we let $S = \{\ell_1, \ldots , \ell_s\}$ denote a finite, non-empty set of $s$ distinct prime numbers.

\begin{definition} 
Let $n$ be a nonzero integer and write $n = A \ell_1^{r_1} \ldots \ell_s^{r_s}$, where 
$r_1, \ldots , r_s$ are non-negative integers and $A$ is an integer 
relatively prime to $\ell_1 \ldots \ell_s$. We define the $S$-part $[n]_S$ 
%and the $S$-free part 
of $n$ by 
$$
[n]_S := \ell_1^{r_1} \ldots \ell_s^{r_s}.
% \quad \hbox{and} \quad (n)_S := |M|.
$$
We set $[0]_S = 0$. 
\end{definition}

Let $\cN$ denote a sequence of integers, defined in some natural way. 
We discuss whether one can improve the trivial estimate $[n]_S \le n$ 
for integers $n$ in $\cN$. 
When this is the case, then, by taking for $S$ the set composed of the first $s$ prime numbers, 
this often implies a lower bound for the greatest 
prime factor of $n$ in $\cN$. 
%We do this twice to illustrate. 
%However, we may be able to derive 
%such a lower bound, without having an $\eps$-improvement on $[n]_S$; see e.g. \cite{BuKa18}. 

We consider several different types of sets $\cN$ and survey various recent results obtained in 
\cite{Bu18a,Bu21,Bu22,BuEv17,BuEvGy18,BuKa18}. 
In this section, all the theorems with an arbitrary, positive $\eps$ in their statement are proved by means of the 
Schmidt Subspace Theorem and are all ineffective, while the corresponding effective statements 
involve a positive $\tau$ and their proofs depend on the theory of linear forms in logarithms and, in 
a crucial way, on the term $B'$. 

We postpone to Section \ref{sec:6} new results on the $S$-part of sequences of convergents to real numbers. 

\subsection{Integers with few digits in an integer base}

%Let $b$ denote an integer at least equal to $2$. 
For integers $b \ge 2$ and $k \ge 2$, we denote by $(u_j^{(b, k)})_{j \ge 1}$ 
the sequence, arranged in increasing order, of all positive integers which are 
not divisible by $b$ and have at most $k$ nonzero digits in their representation in base $b$. 
Said differently, $(u_j^{(b, k)})_{j \ge 1}$ is the ordered sequence composed of the integers 
$1, 2, \ldots , b-1$ and those of the form
$$
d_k b^{n_k} + \ldots + d_2 b^{n_2} + d_1
$$
with
$$
n_k > \cdots > n_2 > 0, \quad  
0 \le d_1, \ldots , d_k \le b-1, \quad
d_1 d_k \not= 0.
$$
The following result reproduces \cite[Theorem 1.1]{Bu18a} and \cite[Theorem 1.5]{Bu21}. It extends \cite[Theorem 1.2]{Bu18a}, 
which deals with the case $k=3$.

\begin{theorem}   \label{th:ubk}
Let $b \ge 2, k \ge 2$ be integers and $\eps$ a positive real number. 
%Let $b \ge 2$ be an integer and $\eps$ a positive real number. 
Let $S$ be a finite, non-empty set of prime numbers. 
Then, we have 
$$
[u_j^{(b, k)}]_S <  (u_j^{(b, k)})^{\eps}, 
$$
for every sufficiently large integer $j$. 
In particular, the greatest prime factor of $u_j^{(b, k)}$ tends to infinity as $j$ tends to
infinity. 
Furthermore, there exist effectively computable positive numbers $j_0$ and $\tau$, depending 
only on $b, k$, and $S$, such that
$$
[u_j^{(b, k)}]_S \le (u_j^{(b, k)})^{1 - \tau}, \quad
\hbox{for every $j \ge j_0$}.
$$
\end{theorem} 

The dependence of $\tau$ on $b, k, S$ is made explicit in \cite{Bu21}, up to some absolute 
numerical constants. 
As a consequence of the main result of \cite{BuKa18}, 
we get that, for $k \ge 3$, there exists an effectively computable positive integer $n_0$, 
depending 
only on $b,k$ and $\eps$, such that any integer $n > n_0$ which is 
not divisible by $b$ and has all of its prime factors less than 
$$
\Bigl({1\over k-2} - \eps\Bigr) (\log \log n) {\log \log \log n \over \log \log \log \log n} 
$$
has at least $k+1$ nonzero digits in its $b$-ary representation. The term $\frac{1}{k-2}$ can be replaced 
by $\frac{1}{k-1}$, see \cite{Bu21}.

%Maybe give the estimate for $P[ . ]$?

A version of Theorem \ref{th:ubk} for the Fibonacci numeration system (Zeckendorf expansions) has been 
established in \cite{Bu21}. 
We point out that the proofs in \cite{Bu21} depend only on Theorem \ref{lflog}, while Theorems \ref{lflog} 
and \ref{Yu} are combined in \cite{Bu18a,BuKa18}.

\subsection{Recurrence sequences of integers}

Let $k$ be a positive integer, and let $a_1, \ldots , a_k$ and $u_0, \ldots , u_{k-1}$ be integers
such that $a_k$ is non-zero and $u_0, \ldots , u_{k-1}$ are not all zero. Put
\begin{equation}  \label{un}
u_n = a_1 u_{n-1} + \ldots + a_{k} u_{n-k}, \quad \hbox{for $n \ge k$}.     
\end{equation}
The sequence $(u_n)_{n \ge 0}$ is a linear recurrence sequence of integers of order $k$. Its 
characteristic polynomial 
$$
G(X) := X^k - a_1 X^{k-1} - \ldots - a_k
$$
factors as 
$$
G(X) = \prod_{i=1}^t \, (X - \alpha_i)^{h_i},
$$
where $\alpha_1, \ldots , \alpha_t$ are distinct algebraic numbers 
with $|\alpha_1| \ge |\alpha_2| \ge \ldots \ge |\alpha_t|$ and $h_1, \ldots , 
h_t$ are positive integers. 
The recurrence sequence $(u_n)_{n \ge 0}$ is said to be {\it degenerate} if there are 
integers $i, j$ with $1 \le i < j \le t$ such that $\alpha_i / \alpha_j$ is a root of unity.
%The sequence $(u_n)_{n \ge 0}$ 
It is said to have a dominant root if $|\alpha_1| > |\alpha_2|$. 

Choose embeddings of $\Q(\alpha_1, \ldots , \alpha_t)$ in $\C$ and of $\Q(\alpha_1, \ldots , \alpha_t)$ in $\Q_p$, 
for every prime $p$.
These embeddings define extensions to $\Q(\alpha_1, \ldots , \alpha_t)$ of the ordinary absolute value $|\cdot |$
and of the $p$-adic absolute value $|\cdot |_p$ for every prime $p$, 
normalized such that $|p|_p = p^{-1}$.

\begin{theorem}    \label{srl}
Let $(u_n)_{n \ge 0}$ be a non-degenerate recurrence sequence of integers defined in \eqref{un}. 
Let  $S := \{\ell_1, \ldots , \ell_s\}$ be a finite, non-empty set of prime numbers, and set
$$
\delta :=-{\sum_{i=1}^s \log\max \{|\alpha_1|_{\ell_i}, \ldots ,  |\alpha_t|_{\ell_i}\} \over
\log\max \{|\alpha_1|, \cdots , |\alpha_t|\} }.   
$$
Let $\eps >0$. Then, we have
$$
|u_n|^{\delta -\eps}\leq [u_n]_S\leq |u_n|^{\delta +\eps}, 
$$ 
for every sufficiently large $n$. 
In particular, if $\gcd (\ell_1\cdots \ell_s,a_1, \ldots ,  a_k)=1$, then we have 
$$
[u_n]_S\leq |u_n|^{\eps}, 
$$
for every sufficiently large $n$.
%\end{theorem}
%The next theorem only applies to a special class of non-degenerate recurrence sequences, namely to 
%those having a dominant root. 
%
%\begin{theorem}
%Let $(u_n)_{n \ge 0}$ be a non-degenerate recurrence sequence of integers 
%having a dominant root. 
%Let $S$ be a finite, non-empty set of prime numbers. 
Furthermore, if $(u_n)_{n \ge 0}$ has a dominant root,
then there exist effectively computable positive numbers $n_0$ and $\tau$, depending 
only on $(u_n)_{n \ge 0}$ and $S$, such that
$$
[u_n]_S \le |u_n|^{1 - \tau}, \quad
%\hbox{for every $n \ge c_2$}.
$$
for every $n \ge n_0$. 
\end{theorem}

Removing the dominant root assumption in the last statement of Theorem \ref{srl} 
seems to be very difficult. However, this can be 
done for non-degenerate binary recurrence sequences of integers; see \cite{BuEv17} for 
references.

\subsection{Polynomials, binary forms, and decomposable forms} 

The results of this subsection have been established in \cite{BuEvGy18}; see the references 
therein for earlier works. 

\begin{theorem}\label{thm1a.1}
Let $f(X)$ be an integer polynomial of degree $d \geq 2$. 
Let $S$ be a non-empty set of prime numbers. 
If $f(X)$ has no multiple zeros, then, for every $\eps >0$ and
for every integer $n$, 
%with $f(n)\not= 0$, 
we have 
\[
[f(n)]_S\ll_{f,S,\eps} |f(n)|^{(1/d)+\eps}. 
\]
If $f(X)$ has at least two distinct roots, then,  
for every integer $n$, 
%with $f(n)\not= 0$, 
we have
\[
[f(n)]_S \ll_f  |f(n)|^{1-\tau},
\]
where $\tau$ is an effectively computable positive number that depends
only on $f(X)$ and $S$.
\end{theorem}

Bennett, Filaseta,    
and Trifonov \cite{BeFiTr08,BeFiTr09} have obtained stronger effective results for the 
polynomials $X(X+1)$ and $X^2 + 7$ and special sets $S$. 

Note that  there are infinitely many primes $p$, and for each of these
$p$, there are infinitely many integers $n$, such that
$f(n)\not= 0$ and 
\[
[f(n)]_{\{ p\}}\gg_f |f(n)|^{1/d}.
\]

We now formulate an analogue of Theorem \ref{thm1a.1} for
binary forms. Denote by $\zprim$ the set of pairs $(x,y)$ in $\Z^2$ with
$\gcd (x,y)=1$.

\begin{theorem}\label{thm1b.1}
Let $F(X,Y)$ be a binary form of degree $d \geq 2$ with integer coefficients.
Let $S$ be a non-empty set of primes. 
If the discriminant of $F(X,Y)$ is nonzero, then, 
for every $\eps >0$ and every pair $(x,y)$ in $\zprim$, 
%with $F(x,y)\not= 0$, 
we have
\[
[F(x,y)]_S\ll_{F,S,\eps} |F(x,y)|^{(2/d)+\eps}.
\]
If $F$ has at least three pairwise non-proportional linear factors over its splitting field, then
\[ 
[F(x,y)]_S \ll_F |F(x, y)|^{1-\tau}       
\]
for every $(x,y)$ in $\zprim$, 
%with $F(x,y)\not=0$,
where 
$\tau$ is an effectively computable positive number, 
depending only on $F$ and $S$.
\end{theorem}

Note that there are finite sets of primes $S$ with 
the smallest prime in $S$ being arbitrarily large,
and, for each one of these sets $S$, infinitely many pairs $(x,y)$ in $\zprim$, such that
$F(x,y)\not= 0$ and
\[
[F(x,y)]_S\gg_{F,S} |F(x,y)|^{2/d}.
\] 

Theorem \ref{thm1b.1} extends to a class of decomposable form equations, see \cite{BuEvGy18}.

\subsection{Power sums}

Apparently, arithmetical properties of the sequence of integers of the form $2^m + 6^n + 1$ were first 
discussed by Corvaja and Zannier in \cite{CoZa05}, where they showed, by a 
clever use of the Schmidt Subspace Theorem, that it contains only finitely many squares. 
We prove a result related to a special case of \cite[Theorem 4.3]{Bu18a}, namely 
we bound from above the $S$-part of $2^m + 6^n + 1$ by applying 
Theorems \ref{lflog} and \ref{Yu}.

\begin{theorem} \label{2m6n} 
Let $S$ be a finite, non-empty set of prime numbers. 
Let $a, b$ be integers with $a \ge 2, b \ge 2$. 
Then, for every $\eps > 0$, we have
$$
[a^m + b^n + 1]_S \le (a^m + b^n + 1)^\eps, 
$$
if $m+n$ is sufficiently large. 
Furthermore, if $\gcd(a, b) > 1$, then there exist  
effectively computable positive numbers $\tau$ and $n_0$, depending only on $a, b$, and $S$, such that 
$$
[a^m + b^n + 1]_S \le (a^m + b^n + 1)^{1 - \tau}, \quad \hbox{for $m+n \ge n_0$.}
$$
\end{theorem}

\begin{proof} 
The first statement easily follows from the $p$-adic Schmidt Subspace Theorem. 
We omit the proof. 

We treat briefly the special case $a=2$, $b=6$. The general case goes along the same lines. Write
$$
2^m + 6^n + 1 = A \ell_1^{r_1} \ldots \ell_s^{r_s},
$$
with $\gcd(A, \ell_1 \ldots \ell_s) = 1$. Without loss of generality, we can assume 
that $A \le (2^m + 6^n + 1)^{1/2}$. Then, we have
$$
\max\{m, n\} \ll_S \max\{r_1, \ldots , r_s\} \ll_S \max\{m, n\}.
$$
We distinguish three cases. 

If $2^m \ge 6^{2n}$, then $A \ell_1^{r_1} \ldots \ell_s^{r_s} 2^{-m}$ is very close to $1$ and it follows from 
\eqref{eq4} that 
$$
m \ll_S \log (2A) \, \log \frac{m}{\log (2A)}.
$$
If $6^n \ge 2^{2m}$, then $A \ell_1^{r_1} \ldots \ell_s^{r_s} 6^{-n}$ is very close to $1$ and it follows from 
\eqref{eq4} that 
$$
n \ll_S \log (2A) \, \log \frac{n}{\log (2A)}.
$$
In both cases, we get
$$
\max\{m, n\} \ll_S \log (2A). 
$$
It remains for us to treat the case where
$2^m \le 6^{2n} \le 2^{4m}$. Then, $\min\{m, n\} \gg \max\{m, n\}$ and $A \ell_1^{r_1} \ldots \ell_s^{r_s}$ is $2$-adically very close to $1$. 
We derive from \eqref{eq8} that
$$
\max\{m, n\} \ll_S \log (2A) \, \log \frac{\max\{m, n\} }{\log (2A)}. 
$$
This completes the proof of the theorem. 
\end{proof}

A key ingredient in the proof of the second part of Theorem \ref{2m6n} is the assumption that $a$ and $b$ are not coprime. 
It allows us to combine Archimedean and non-Archimedean estimates. 
We do not see how to get a non-trivial upper bound for the $S$-part of $1 + 2^m + 3^n$. 
Difficulties occur when the integers $m, n$ are such that $2^m$ is close to $3^n$. 

\begin{problem}
For a prime number $p \ge 5$, give a non-trivial upper bound for 
$$
[1 + 2^m + 3^n]_{\{p \}}. 
$$
\end{problem}

\subsection{Sums of integral $T$-units}

Let $T$ be a finite set of prime numbers, whose intersection with $S$ is empty. 
We consider the set $\cN$ composed of all the integers whose prime divisors are in $T$. 

\begin{theorem}
Let $\eps$ be a positive real number. 
Let $T = \{q_1, \ldots , q_t\}$ be a finite set of prime numbers. 
Let $x, y$ be coprime integers whose prime divisors are in $T$. 
If $|x+y|$ is sufficiently large in terms of $\eps$, then
$$
[x + y]_S \le |x + y|^{\eps}.
$$
Furthermore, there exist effectively computable positive numbers $\tau$ and $X_0$, depending only on $S$ and $T$,  
such that
$$
[x + y]_S \ll_{S,T} |x+y|^{1 - \tau},
$$
if $|x + y| > X_0$. 
\end{theorem}

The first assertion is an easy consequence of Ridout's theorem \cite{Rid57}. 
It extends to sums of an arbitrary number of $S$-units, by means of the $p$-adic Subspace Theorem. 
The second assertion does not seem to have been stated previously. 
%It is not new, since it is contained in \cite{BEG09}.  {\bf [Est-ce vraiment le cas ?]} 

\begin{proof}
Write $x = q_1^{u_1} \ldots q_t^{u_t}$, $y = q_1^{v_1} \ldots q_t^{v_t}$, and 
$$
q_1^{u_1} \ldots q_t^{u_t} \pm q_1^{v_1} \ldots q_t^{v_t} = A \ell_1^{w_1} \ldots \ell_s^{w_s}, 
$$
where the integer $A$ is coprime with $\ell_1 \ldots \ell_s$. 
Assume that $|x + y| \ge 4$ and $|A| \le |x + y|^{1/2}$. 
Set 
$$
B = \max\{u_1, \ldots , u_t, v_1, \ldots , v_t\}, \quad 
W = \max\{w_1, \ldots , w_s\}, 
$$
and $A^* = \max\{|A|, 2\}$. 
Throughout the proof, the constants implicit in $\ll$ are effectively computable 
and depend at most on
$S$ and on $T$. Our assumption on $A$ implies that 
$$
B \gg \ll W. 
$$
Assume that $B = u_1$ and consider 
$$
\vv_{q_1} (A \ell_1^{w_1} \ldots \ell_s^{w_s} \pm q_1^{v_1} \ldots q_t^{v_t}). 
$$
It follows from Theorem \ref{Yu} that
$$
B \ll (\log A^*) \, \log \frac{B+W}{\log A^*}, 
$$
%Since $W \ll B$, we derive that 
thus,
$$
B \ll \log A^*. 
$$
Setting $X = \max\{|x|, |y|, 2\}$, we get
$$
\log X \ll B \ll \log A^*,
$$
thus $X \ll 1$ if $|A| = 1$ and, if $|A| \ge 2$, then there exist positive effectively computable numbers 
$c_1, c_2, c_3$, depending on $S$ and $T$, such that 
$$
|A| \gg 2^{c_1 B} \gg 2^{c_2 W} \gg |x + y|^{c_3}. 
$$
This completes the proof. 
\end{proof} 

\subsection{Summary} 
In all the examples above, there is a big gap between ineffective and effective statements. 
The term $B'$ allows us to get a small, effective $\tau$-saving on the trivial bound.

\section{$S$-unit equations}   \label{sec:4}

Let $K$ be an algebraic number field.  
Many Diophantine problems reduce to equations of the form
$$
a_1 x_1 + a_2 x_2 = 1,
$$
where $a_1, a_2$ are given elements of $K$ and the unknowns $x_1, x_2$ are elements 
of its unit group $O_K^*$ or, more generally, of a group of $S$-units in $K$, where $S$ 
is a finite set of places on $K$ containing all the infinite places. 
Recall that an algebraic number $x$ in $K$ is an $S$-unit if, by definition, 
$|x|_v = 1$ for every place $v$ not in $S$. 

These equations are called {\it unit equations} and {\it $S$-unit equations}; see 
the monograph \cite{EvGy15} for explicit estimates of 
the height of their solutions and many bibliographic references.
%Explicit upper bound for the solutions of unit and $S$-unit equations were given 
%for the first time by Gy\H ory \cite{Gy74}; see also Bugeaud and Gy\H ory \cite{BuGy96a},
%Gy\H ory and Yu \cite{GyYu06} and 
%the monographs \cite{EvGy15,EvGy16} by Evertse and Gy\H ory, where the
%reader can find many bibliographic references and numerous applications. 

\begin{theorem}   \label{th:Sunit0}
Let $K$ be an algebraic number field of degree $d$ and discriminant $D_K$. 
Let $S$ be a finite set of places on $K$ 
containing the infinite places.
Let $a_1, a_2$ be non-zero elements of $K$.
The equation
$$
a_1 x_1 + a_2 x_2 = 1, \quad \hbox{in $S$-units $x_1, x_2$ in $K$,}
$$
has only finitely many solutions, and all of them satisfy
$$
{\rm max}\{h(x_1), h(x_2)\} \ll_{d, D_K, S}  {\rm max} \{h(a_1 ), h(a_2 ), 1\}. 
$$
\end{theorem}

The fact that the upper bound in \eqref{th:Sunit0} is linear in ${\rm max} \{h(a_1 ), h(a_2 ), 1\}$ is a 
consequence of term $B'$ in \eqref{eq2}. Using \eqref{eq1} instead, with $B$ in place of $B'$, 
we would get an extra factor $\log {\rm max} \{h(a_1 ), h(a_2 ), 2\}$ in the upper bound.

\subsection{Linear equations in two unknowns from a multiplicative division group}  

More generally, we can take an arbitrary finitely generated subgroup $\Gamma$ of positive rank of the multiplicative 
group $(\ovQ)^2 = \ovQ \times \ovQ$, endowed with coordinatewise multiplication, and consider
\begin{equation}  \label{SunitBEG}
a_1 x_1 + a_2 x_2 = 1, \quad
\hbox{in $(x_1, x_2) \in \Gamma$,}
\end{equation}
where $a_1, a_2$ are given nonzero complex algebraic numbers. 
We can compute explicit upper bounds for the heights of $x_1$ and $x_2$. 
We quote below important results from \cite{BEG09}, which deal with extensions of \eqref{th:Sunit0}. 

For $\bfx = (x_1, x_2)$ in $(\ovQ)^2$, define $h(\bfx) = h(x_1) + h(x_2)$. 
The division group of $\Gamma$ is the set
$$
\bG = \{ \bfx \in (\ovQ)^2 : \bfx^k \in \Gamma \, \, \hbox{for some $k$ in $\Z_{> 0}$} \}
$$
For $\eps > 0$, the cylinder around $\bG$ and the truncated cone around $\bG$ are given by 
$$
\bG_\eps = \{  \bfx \in (\ovQ)^2 : \exists \, \bfy, \bfz \hbox{ with } \bfx = \bfy \bfz, \bfy \in \bG, \bfz \in (\ovQ)^2, h(\bfz) < \eps\}
$$
and
$$
C(\bG, \eps) = \{  \bfx \in (\ovQ)^2 : \exists \, \bfy, \bfz \hbox{ with } \bfx = \bfy \bfz, \bfy \in \bG, \bfz \in (\ovQ)^2, 
h(\bfz) < \eps (1 + h(\bfy)) \},
$$
respectively. 
We stress that the points in $\bG, \bG_\eps$ or $C(\bG, \eps)$ do not have their coordinates in a prescribed number field.

Nevertheless, it is possible to bound the height of a solution $(x_1, x_2)$ in $(\ovQ)^2$ to the equation 
$$
a_1 x_1 + a_2 x_2 = 1
$$ 
which belongs to $\Gamma$, to $\bG_\eps$ or to $C(\bG, \eps)$, provided that $\eps$ is sufficiently small. 
Below, $K_0$ denotes the algebraic number field generated by $a_1, a_2$, and the elements of $\Gamma$,
%$$
%K_0 = \Q (a_1, a_2, \Gamma),
%$$
and $h_0$ is an upper bound for the heights of the components of a system of generators of $\Gamma$ 
modulo torsion. Furthermore, $r$ is the rank of $\Gamma$ and $A$ is an effectively computable 
positive real number, which depends only on $\Gamma$. The reader is referred to \cite{BEG09} 
for an explicit expression of $A$. 

\begin{theorem}  \label{thBEG}
Every solution $(x_1, x_2)$ in $\Gamma$ of \eqref{SunitBEG} satisfies 
$$
h(x_1, x_2) < A \max\{h(a_1), h(a_2), 1\}. 
$$
Suppose that $(x_1, x_2)$ is a solution to
$$
a_1 x_1 + a_2 x_2 = 1, \quad
\hbox{in $\bfx = (x_1, x_2)$ in $\bG$ or in $\bG_\eps$ with $\eps < 0.0225$.}
$$
%and that 
%$$
%\eps < 0.0225.
%$$
Then we have
$$
h(\bfx) \le A h(a_1, a_2) + 3 r h_0 A
$$
and
$$
[K_0 (x_1, x_2) : K_0] \le 2.
$$
Suppose that $(x_1, x_2)$ is a solution to
$$
a_1 x_1 + a_2 x_2 = 1, \quad
\hbox{in $\bfx = (x_1, x_2)$ in $C(\bG, \eps)$,}
$$
and that 
$$
\eps < \frac{0.09}{8 A h(a_1,a_2) + 20 r h_0 A}. 
$$
Then we have
$$
h(\bfx) \le 3 A h(a_1, a_2) + 5 r h_0 A
$$
and
$$
[K_0 (x_1, x_2) : K_0] \le 2.
$$
\end{theorem}

Here also, the fact that the upper bounds for $h(\bfx)$ are linear in $h(a_1, a_2)$ is a 
consequence of term $B'$ in \eqref{eq2}. 
A key auxiliary result in the proof of Theorem \ref{thBEG} is a lemma 
of Beukers and Zagier \cite{BeZa97}, which gives a lower bound for the sum of the 
heights of three distinct algebraic points with non-zero coordinates lying on a line 
$\lambda x + \mu y + \nu z = 0$ with $\lambda \mu \nu = 0$.

\subsection{Unit equations on quaternions}

%Maybe only a few words. 

Let $\HH = \R \oplus \R \rmi \oplus \R \rmj \oplus \R \rmk$ denote the quaternion algebra $\HH$ over $\R$, 
with the standard multiplication law $\rmi^2 = \rmj^2 = \rmk^2 = -1$, $\rmi \rmj = -\rmj \rmi = \rmk$, 
$\rmj \rmk = -\rmk \rmj = \rmi$, $\rmk \rmi = -\rmi \rmk = \rmj$. For an element $\alpha = a + b \rmi +
c \rmj + d \rmk$ in $\HH$, where $a, b, c, d$ are in $\R$, define its conjugate to be 
$\alpha = a - b \rmi - c \rmj - d \rmk$, its norm to be
$$
N(\alpha) = \alpha \alphabar = \alphabar \alpha = a^2 + b^2 + c^2 + d^2,
$$ 
and its trace 
$$
\tr(\alpha) = \alpha + \alphabar = 2a. 
$$
Write $|\alpha| = \sqrt{N(\alpha)}$.
We say that a quaternion $\alpha = a + b \rmi + c \rmj + d \rmk$ in $\HH$ is algebraic if all its coordinates $a, b, c, d$ 
are algebraic over $\Q$. This is equivalent to requiring that $\alpha$ satisfies a polynomial equation with coefficients
in $\Q$, or that $\Q[\alpha]$ is a finite field extension of $\Q$. Indeed, $\alpha$ always satisfies the quadratic equation
$$
X^2 - \tr(\alpha) X + N(\alpha) = 0
$$
and if $a, b, c, d$ are in $\Q$, then so are $\tr(\alpha)$ and $N(\alpha)$. 

Denote by $\HH_a^\times$ the subalgebra of all quaternions that are algebraic. The next statement 
has been proved by Huang \cite[Theorem 1.2]{Hu20}. 

\begin{theorem} 
Let $\Gamma_1, \Gamma_2$ be semigroups of $\HH_a^\times$ 
generated by finitely many elements of norms greater
than $1$, and fix $a, a', b, b'$ in $\HH_a^\times$. 
Then the equation
$$
afa' + bgb' = 1
$$
has only finitely many solutions $(f, g)$ in $\Gamma_1 \times \Gamma_2$ such that 
$|1 - afa'| \not= |afa'|$. 
Moreover, for such pairs $( f, g)$, we have effectively computable upper bounds for $| f |$, $|g|$ 
that depend only on $a, a', b, b'$
and generators of $\Gamma_1, \Gamma_2$. 
\end{theorem}

By inspecting the proof in \cite{Hu20}, we see that an application of \eqref{eq2} (instead of \eqref{eq1} used in \cite{Hu20}) 
yields an upper bound for $|f|, |g|$ which is linear in $\max\{|a|, |a'|, |b|, |b'|\}$.

\section{Almost powers in integer sequences}  \label{sec:5}

As in Section \ref{sec:3}, we let $\cN$ be a sequence of positive integers defined in some natural way.  
Assume that the theory of linear forms in logarithms allows us to compute an upper bound for the prime number $q$ 
such that $y^q$ is in $\cN$ for some integer $y \ge 2$. This is for example the case when
$$
\cN = \{ n(n+1) : n \ge 1\} \quad \hbox{and} \quad 
\cN = \{2^n + 3^n : n \ge 1\}.
$$
In some situations, the proof can be modified to take advantage of $B'$ 
in order to show that, more generally, $A y^q$ cannot be in $\cN$ for $y \ge 2$, a prime number
$q$ sufficiently large and a rational number $A$ of sufficiently small height in terms of~$q$.

We consider three examples below. Other interesting equations, which can be treated in a 
similar way, include $2^m + 2^n + 1 = A y^q$ and sums of two integral $T$-units 
being almost powers, see \cite{BeBi17}.

\subsection{Almost powers in values of polynomials}

The next result has been established in \cite{BBMOS}.

\begin{theorem}
Let $f(X)$ be an integral polynomial with at least two distinct roots. 
If there are integers $x, y, q$ and a rational number $A$ such that $|y| \ge 2$ and 
$$
f(x) = A y^q,
$$
then 
$$
q \ll_f \max\{1, h(A)\}. 
$$
\end{theorem}

By using \eqref{eq1} instead of \eqref{eq2} in the proof, one gets the weaker estimate 
$q \ll_f \max\{1, h(A)\} \log \max\{2, h(A)\}$. 

%A more general result is established in \cite{BBMOS}. 

\subsection{Almost powers in recurrence sequences}

The next statement deals with recurrence sequences having a dominant root. 

\begin{theorem}
Let ${\bf u} := (u_n)_{n \ge 0}$ be a recurrence sequence of integers given by
$$
u_n = f_1(n) \alpha_1^n + \ldots + f_t(n) \alpha_t^n, \quad \hbox{for $n \ge 0$},
$$
and having a dominant root $\alpha_1$. Assume that $\alpha_1$ is a simple root, that is,
the polynomial $f_1(X)$ is equal to a non-zero algebraic number $f_1$. 
Then, the equation
$$
u_n = A y^q,  
$$
in rational $A$ and integers $n, y, q$ with $q$ a prime number, 
$|y| \ge 2$, and $u_n \not= f_1 \alpha_1^n$, implies that 
$$
q \ll_{{\bf u}} \, \max\{1, h(A)\}. 
$$
\end{theorem}

\begin{proof}
We establish a slightly more general result. 
We consider a sequence of nonzero integers ${\bf v} := (v_n)_{n \ge 0}$ 
with the property that there are $\theta$ in $(0, 1)$ and 
a positive real number $C$ such that
$$
| v_n - f \alpha^n | \le C \, |\alpha|^{\theta n},  \quad n \ge 0, 
$$
where $f$ is a non-zero algebraic number
and $\alpha$ is an algebraic number with $|\alpha|>1$.

%Let $A$ be a nonzero rational number and $A^*$ such that $\log A^* = \max\{h(A), 2\}$. 
Let $n, y, q$ with $q$ a prime number 
be such that $v_n = A y^q$, $|y| \ge 2$, and $v_n \not= f \alpha^n$. Then, we have 
\begin{equation}  \label{eq3.10} 
|A f^{-1} \alpha^{-n} y^q - 1| \le {C \over |f|}  \, |\alpha|^{(\theta- 1) n}.   
\end{equation}
There exist integers $k$ and $r$ such that $n = k q + r$ with $|r| \le {q \over 2}$. 
Rewriting \eqref{eq3.10} as
$$
\Lambda := \Bigl| A f^{-1} \Bigl({y \over \alpha^k} \Bigr)^q \alpha^{-r} - 1 \Bigr| 
\le {C \over |f|} \, |\alpha|^{(\theta- 1) n}, 
$$
we observe that 
\begin{equation}  \label{eq3.11} 
\log \Lambda \ll_{\bf v} \, (- n ).  
\end{equation}
The height of ${\alpha^k \over y}$ is bounded from above by the sum of 
$\log |y|$ and $k$ times the height of $\alpha$. Since
\begin{equation}  \label{eq3.12} 
q \log |y|  \ll_{\bf v} \,  n   \le  (k+1) q
\quad \hbox{and} \quad
{k q \over 2} \le  n   \ll_{\bf v} \, q \log| y|,     
\end{equation}
we get that $h(\alpha^k / y) \ll_{\bf v} \, \log |y|$. 
Set $h^* (A) = \max\{1, h(A)\}$. 
It then follows from \eqref{eq4} that
$$
\log \Lambda \gg_{\bf v} \,  - h^* (A) \, (\log |y|) \Bigl(\log \frac{q}{h^* (A)} \Bigr),
$$
which, combined with \eqref{eq3.11} and \eqref{eq3.12}, gives 
$$
q \log |y| \ll_{\bf v} \, n \ll_{\bf v} \, h^* (A) \, (\log |y|) \Bigl(\log \frac{q}{h^* (A)} \Bigr),
$$
and we get $q \ll_{\bf v} \, h^* (A)$. This proves the theorem.    
Note that \eqref{eq5} does not enable us to conclude. 
\end{proof}

\subsection{Almost powers in power sums}   \label{appsum}
%Let $S = \{p_1, \ldots , p_s\}$ be a finite set of prime numbers. 
We provide another result on Diophantine properties of the sequence 
of integers of the form $2^m + 6^n + 1$. 

\begin{theorem} 
All the solutions to the Diophantine equation
$$
%2^m + 6^n + 1 = \frac{A}{A'}  p_1^{r_1} \ldots p_s^{r_s} y^q. 
2^m + 6^n + 1 =  A  y^q,
$$
in positive integers $m, n, q$ and $A$ rational, satisfy
$$
q \ll \max\{1, h(A)\}. 
$$
\end{theorem}

\begin{proof}
As in the proof of Theorem \ref{2m6n}, 
we distinguish three cases. If $2^m \ge 6^{2n}$ or if $6^n \ge 2^{2m}$, then
\eqref{eq4} implies that 
$$
q (\log y) \ll \max\{m, n\} \le h^* (A) (\log y) \log \frac{q}{h^* (A)}, 
$$
where $h^* (A) = \max\{1, h(A)\}$. 
We conclude that $q \ll \log h^*(A)$. 

Now, if $2^m \le 6^{2n} \le 2^{4m}$, then $m \gg \ll n$ and
Theorem \ref{BuLa} gives that 
$$
q \log y \ll m \ll h^* (A) \, (\log y) (\log (q / h^* (A)))^2.
$$
In both cases, we get that $q \ll h^* (A)$, as asserted. 
\end{proof}

We may try to get one step further. 
Let $p \ge 5$ be a prime number and consider the equation
$$
%2^m + 6^n + 1 = \frac{A}{A'}  p_1^{r_1} \ldots p_s^{r_s} y^q. 
2^m + 6^n + 1 =  A  p^{\ell} y^q,
$$
Without any restriction we assume that $\ell$ satisfies $|\ell| < q$. 
The cases $2^m \ge 6^{2n}$ and $6^n \ge 2^{2m}$ go as above and yield that 
$$
q \ll  (\log p) (\log \log p) h^* (A). 
$$
Assume now that $2^m \le 6^{2n} \le 2^{4m}$. 
Theorem \ref{Yu} gives that either
\begin{equation}  \label{eqSix}
q \ll h^* (A) (\log p) (\log y),
\end{equation} 
or
$$
q \log y \ll m \ll h^* (A) (\log p) (\log y) (\log (q / h^* (A))).
$$
Unfortunately, we cannot deduce anything from \eqref{eqSix}, when $y$ is assumed to be variable. 
If one prefers to apply \eqref{eq8}, then one gets
$$
q \log y \ll m \ll h^* (A) (\log p) (\log y) (\log (q (\log y)/ h^* (A)),
$$
which does not yield a bound for $q$ independent of $y$.

\section{Arithmetical properties of convergents}  \label{sec:6}

In this section, 
$\theta$ is an arbitrary irrational, 
real number and $(p_n (\theta) / q_n (\theta))_{n \ge 1}$ 
(we will use the shorter notation $p_n / q_n$ when there is no confusion possible 
and $\xi$ instead of $\theta$
if the number is known to be algebraic) denotes
the sequence of its convergents. In particular, we have 
$$
\Bigl| \theta - \frac{p_n (\theta)}{q_n(\theta)} \Bigr| < \frac{1}{q_n (\theta)^2}, \quad n \ge 1.
$$
Recall that, by Legendre's theorem, every rational number $p/q$ 
such that $|\theta - p/q| < 1 / (2 q^2)$ is a convergent to $\theta$.

\begin{theorem}\label{ShoS}
Let $\eps$ be a positive real number.
Let $S$ be a finite, non-empty set of distinct prime numbers. 
Let $(p_n / q_n)_{n \ge 1}$ denote the sequence of convergents to an irrational 
real algebraic number $\xi$. 
Then, we have 
$$
[p_n]_S <  p_n^{\eps},  \qquad [q_n]_S <  q_n^{\eps}, 
$$
for every sufficiently large integer $n$. 
%Consequently, both $P[p_n]$ and $P[q_n]$ tend to infinity with $n$.
Furthermore, there exist effectively computable positive numbers $\tau$ and $n_0$, depending 
only on $\xi$ and $S$, such that
$$
[p_n q_n]_S \le (p_n q_n)^{1 - \tau}, \quad \hbox{for $n\ge n_0$}.
$$ 
Moreover, we have
$$
P[p_n q_n] \gg_\xi  \log \log q_n \cdot
{\log \log \log q_n \over \log \log \log \log q_n}, \quad n \ge 2. 
$$
\end{theorem}

The first assertion is a direct consequence of Ridout's theorem; see \cite{Sho77}. 
The last one has been established in \cite{Bu09}, thereby slightly improving a result of \cite{Sho77}. 
The second assertion is new and follows directly from Theorem \ref{qqS} below.
No non-trivial effective upper bounds for $[p_n]_S$ and $[q_n]_S$ are known for 
every algebraic number of degree at least $3$. 

Erd\H{o}s and Mahler \cite{ErMa} established that, when $\theta$ is not a Liouville number (that is, when the irrationality 
exponent of $\theta$ is finite),  then 
$P[q_{n-1} q_n q_{n+1}]$ tends to
infinity with $n$; see also \cite{BuKh23}. 
However, their result is not effective.
Using Baker's theory of linear forms in logarithms, Shorey \cite{Sho83}
proved that 
$$
P[q_{n-1} q_n q_{n+1}] \gg_\theta \log \log q_n.
$$
%and$$\log Q[q_{n-1} q_n q_{n+1}] \ggeff \log \log q_n.$$
We can do slightly better by applying Theorem \ref{Yu}.

\begin{theorem} \label{ShoBis}
Let $\theta$ be a real number, which is not a Liouville number. 
Let $S$ be a finite, non-empty set of prime numbers. 
Then, there exist effectively computable positive numbers $\tau$ and $n_0$, depending 
only on $\theta$ and $S$, such that
$$
[q_{n-1} q_n q_{n+1}]_S \le (q_{n-1} q_n q_{n+1})^{1 - \tau}, \quad \hbox{for $n\ge n_0$}.
$$
%, \quad \hbox{for $n\ge n_1$}.
Furthermore, we have 
$$
P[q_{n-1} q_n q_{n+1}] \gg_\theta \log \log q_n \cdot
{\log \log \log q_n \over \log \log \log \log q_n}, \quad n \ge 2.     
$$
\end{theorem}

It follows from the proof of Theorem \ref{ShoBis} that $1 / (c \mu(\theta) \log \mu(\theta))$ 
is a suitable value for $\tau$, where $\mu (\theta)$ is the irrationality exponent 
of $\theta$ (see Definition \ref{IrrExp}) and $c$ a positive, effectively computable, real number 
depending only on $S$. 

We denote by $\lfloor x \rfloor$ the integer part of the real number $x$.
Theorem \ref{ShoS} is a consequence of the following result. 

\begin{theorem}   \label{qqS} 
Let $S$ be a finite, non-empty set of prime numbers. 
For every irrational, real algebraic number $\xi$, 
there exist effectively computable positive numbers $q_0$ and $\tau$, depending 
only on $\xi$ and $S$, such that
$$
[q \lfloor q \xi \rfloor]_S \le (q \lfloor q \xi \rfloor)^{1 - \tau}, \quad \hbox{for $q\ge q_0$}.
$$ 
Furthermore, we have
$$
P[q \lfloor q \xi \rfloor ] \gg_\xi \log \log q \cdot 
{\log \log \log q \over \log \log \log \log q}.
$$
\end{theorem}

\begin{proof} 
Without any loss of generality, we assume that $q$ is large enough.
Set $p = \lfloor q \xi \rfloor$ and observe that $|q \xi - p| < 1$, thus 
$$
0 < |q p ^{-1} \xi - 1| < p^{-1}.
$$
Let $S = \{\ell_1, \ldots , \ell_s\}$ be a finite set of prime numbers with $\ell_1 < \ldots < \ell_s$. 
Let $A_p$ (\resp $A_q$) be the greatest divisor of $p$ (\resp of $q$) coprime 
with $\ell_1 \ldots \ell_s$. 
There are integers
$b_1, \ldots , b_s$ such that
\begin{equation} \label{LaQ}
\Lambda_q := |\xi (A_q / A_p) \ell_{1}^{b_{1}} \ldots \ell_{s}^{b_{s}} - 1| 
\le 1/p.  
\end{equation}
Since $\xi$ is irrational, $\Lambda_q$ is non-zero. 
It then follows from \eqref{eq5} that there exists
an effectively computable constant $c_1$, depending only on $\xi$,
such that
\begin{equation}  \label{LaQ2}
\log \Lambda_q > - c_1^s \, h^*(A) \, \prod_{i=1}^s (\log \ell_i) \, \log {B (\log \ell_s) \over h^*(A)},  
\end{equation}
where $B = \max\{|b_1|, \ldots , |b_s|, 3\}$ and $h^*(A) = \max\{ \log A_q, \log  A_p, 1\}$. 
%The fact that the dependence on $s$ is exponential and not of the form $s^s$ is crucial for this application. The 
%extra factor $\log \ell_s$ does not cause any problem. 

%Assume that $[p q]_S > (p q)^{3/4}$. Then, $\log q < B$ and 
Since $p \ge 2^B$, we derive from \eqref{LaQ} and \eqref{LaQ2} that 
$$
{B \over h^*(A)} \ll \prod_{i=1}^s (c_1 \log \ell_i) \, \log \biggl( \prod_{i=1}^s (\log \ell_i)  \biggr),
$$
thus there exist $\eta > 0$ and $\tau > 0$, depending only on $\xi$, such that 
$$
h^*(A) \ge {\rm e}^{\eta B},
\quad
[pq]_S \le {p q \over h^*(A)}  \le (pq)^{1 - \tau},
$$
when $q$ is large enough. 

In the special case where $\ell_i$ is the $i$-th prime number and $A_p = A_q = 1$, we derive 
from the Prime Number Theorem that
$$
\prod_{i=1}^s (\log \ell_i) \le 3^{s \log \log s},
$$
thus
$$
\log B \ll_\xi s \log \log s,
$$
and
$$
\log \log p q \ll_\xi s \log \log s \ll_\xi  \ell_s \frac{\log \log \ell_s}{\log \ell_s}. 
$$
This gives
$$
P[q \lfloor q \xi \rfloor ] \gg_\xi \log \log q \cdot 
{\log \log \log q \over \log \log \log \log q} ,
$$
as asserted. 
We see the importance of the dependence on $s$ in \eqref{LaQ2}. By using Waldschmidt's estimate 
\eqref{eq2} in place of Matveev's one \eqref{eq5}, 
we would have to replace the factor $s \log \log s$ by $s \log s$
and we would get the (slightly) weaker statement
$$
P[q \lfloor q \xi \rfloor ] \ggeff \log \log q. 
$$
\end{proof}

\begin{proof}[Proof of Theorem  \ref{ShoBis}] 
For $n \ge 2$, set $Q_n = q_{n-1} q_n q_{n+1}$ and
$A_n$ denote the product $A_{-1,n} A_{0,n} A_{1,n}$, 
where $A_{j,n}$ is the greatest divisor of $q_{n+j}$ coprime 
with $\ell_1 \ldots \ell_s$, for $j = -1, 0, 1$. Write also
$$
Q_n = A_n \ell_1^{b_{1,n}} \ldots \ell_s^{b_{s,n}}, \quad 
B_n = \max\{b_{1,n}, \ldots , b_{s,n}, 3\}
$$
and $h^*(A_n) = \max\{\log A_n, 1\}$. 
Observe that 
$$
q_n \quad \hbox{divides} \quad \gcd(Q_n, q_{n+1} - q_{n-1}). 
$$
For every prime $\ell$ dividing $Q_n$, we deduce from Theorem \ref{Yu} that 
$$
B_n \ll_S h^*(A_n)  \quad \hbox{or} \quad v_{\ell} (q_{n+1} - q_{n-1}) \ll_S h^*(A_n)  \,  \log {B_n \over h^*(A_n) },
$$
thus 
$$
B_n \ll_S h^*(A_n)  \quad \hbox{or} \quad  \log q_n - \log A_{0,n} \ll_S  h^*(A_n)  \, \log {B_n \over h^*(A_n) }.
$$
Note that $Q_n \ge 2^{B_n}$, thus
$$
2^{B_n / 3} \le q_{n+1}. 
$$
Since $\theta$ is not a Liouville number, there exists a positive $\mu$ such that 
$q_{n+1} < q_n^\mu$ and we get
$$
\log q_{n+1} < \mu \log q_n \ll_S  \mu h^*(A_n)  \, \log {\log q_{n+1} \over h^*(A_n) }.
$$
This gives a bound for $n$ such that $h^* (A_n) = 1$  
and the existence of $\eta > 0$ and $n_0$, depending only on $\mu$ and $S$, such that 
$$
A_n > q_{n+1}^\eta > Q_n^{\eta / 3}, \quad n > n_0. 
$$
thus
$$
[Q_n]_S \le Q_n^{1 - \eta / 3}, \quad n > n_0. 
$$
%Note that if $A_n = 1$ then we have $\log \log q_{n+1} \ll s \log \log s$. 
We derive the last statement of the theorem by taking for $\ell_1, \ldots , \ell_s$ the first 
$s$ prime numbers and assuming that $A_n = 1$. We omit the details. 
\end{proof}

\vskip 6mm

\noindent{\bf Acknowledgements. } 
This paper is an extended version, along with some new results, of the talk I gave at the 
Number Theory Conference held in Debrecen in July 2022. I would like to warmly thank the 
organizers. I am also very grateful to K\'alm\'an Gy\H ory for his constant encouragement, since
his first invitation to visit the University 
of Debrecen in 1995, while I was a PhD student.


\begin{thebibliography}{99}


\bibitem{Ba64}
A. Baker,
{\it Rational approximations to $\root 3\of 2$ and other algebraic numbers}, 
Quart. J. Math. Oxford Ser. 15 (1964), 375--383.

\bibitem{Ba66}
A. Baker,
%{\it Linear forms in the logarithms of algebraic numbers. I}, 
%Mathematika 13 (1966), 204--216.
{\it Linear forms in the logarithms of algebraic numbers I--IV}, 
Mathematika 13 (1966), 204--216; 14 (1967), 102--107 and 220--224; 15 (1968), 204--216.


\bibitem{Ba72}
A. Baker,
{\it  A sharpening of the bounds for linear forms in logarithms~I},
{ Acta Arith.} {21} (1972), 117--129.


\bibitem{Ba73}
A. Baker,
{\it  A sharpening of the bounds for linear forms in logarithms~II},
{ Acta Arith.} {24} (1973), 33--36.

\bibitem{Ba75}
A. Baker,
Transcendental number theory. 
Cambridge University Press, London-New York, 1975.


\bibitem{BaCo75}
A. Baker and J. Coates,
{\it Fractional parts of powers of rationals}, 
Math. Proc. Cambridge Philos. Soc. 77 (1975), 269--279. 


\bibitem{BaWu93}
A. Baker and G. W\"{u}stholz,
{\it Logarithmic forms and group varieties}, 
J. reine angew. Math. 442 (1993), 19--62.


\bibitem{BaWu07}
A. Baker and G. W\"{u}stholz,
Logarithmic forms and Diophantine geometry. 
New Mathematical Monographs, 9. Cambridge University Press, Cambridge, 2007. 


\bibitem{BaBe02}
M. Bauer and M. Bennett, 
{\it Application of the hypergeometric method to the generalized Ramanujan-Nagell equation}, 
Ramanujan J. 6 (2002), 209--270.

\bibitem{Be93a}
M. A. Bennett,  
{\it Effective lower bounds for the fractional parts of powers of a dense set of rationals}, 
C. R. Math. Rep. Acad. Sci. Canada 15 (1993), 201--206. 

\bibitem{Be93b}
M. A. Bennett,  
{\it Fractional parts of powers of rational numbers},
Math. Proc. Cambridge Philos. Soc. 114 (1993), 191--201. 

\bibitem{Be95}
M. A. Bennett,  
{\it Simultaneous approximation to pairs of algebraic numbers}. 
In: Number theory (Halifax, NS, 1994), 55--65,
CMS Conf. Proc., 15, Amer. Math. Soc., Providence, RI, 1995. 


\bibitem{Be96}
M. A. Bennett,  
{\it Simultaneous rational approximation to binomial functions}, 
Trans. Amer. Math. Soc. 348 (1996), 1717--1738. 



\bibitem{BeBi17}
M. A. Bennett and N. Billerey,
{\it Sums of two $S$-units via Frey-Hellegouarch curves}, 
Math. Comp. 86 (2017), 1375--1401.

\bibitem{BeBu12}
M. A. Bennett and Y. Bugeaud,
{\it Effective results for restricted rational approximation to quadratic irrationals}, 
Acta Arith. 155 (2012), 259--269.

\bibitem{BeFiTr08}
M. A. Bennett, M. Filaseta, and O. Trifonov, 
{\it Yet another generalization of the Ramanujan-Nagell equation}, 
Acta Arith. {\bf 134}, 211--217 (2008).


\bibitem{BeFiTr09}
M. A. Bennett, M. Filaseta, and O. Trifonov, 
{\it On the factorization of consecutive integers}, 
J. Reine Angew. Math. {\bf 629}, 171--200 (2009).
 


\bibitem{BBMOS}
A. B\'erczes, Y. Bugeaud, J. Mello, A. Ostafe, and M. Sha, 
{\it Multiplicative dependence of rational values modulo approximate finitely generated groups}. 
Preprint. 


\bibitem{BEG09}
A. B\'erczes, J.-H. Evertse, and K. Gy\H ory, 
{\it Effective results for linear equations in two unknowns from a multiplicative division group},  
Acta Arith. 136 (2009), 331--349. 


\bibitem{Beu80}
F. Beukers, 
{\it On the generalized Ramanujan-Nagell equation I}, 
Acta Arith. 38 (1980/81), 389--410.


\bibitem{Beu81}
F. Beukers, 
{\it Fractional parts of powers of rationals}, 
Math. Proc. Cambridge Philos. Soc. 90 (1981), 13--20. 


\bibitem{BeZa97}
F. Beukers and D. Zagier, 
{\it Lower bounds of heights of points on hypersurfaces}, 
Acta Arith. 79 (1997), 103--111.  


\bibitem{BiBu00}
Yu. Bilu et Y. Bugeaud,
{\it D\'emonstration du th\'eor\`eme de
Baker-Feldman via les formes lin\'eaires en deux logarithmes}, 
J. Th. Nombres Bordeaux 12 (2000), 13--23. 


\bibitem{BoMu83}
E. Bombieri and J. Mueller, 
{\it On effective measures of irrationality for $\root n \of{a/b}$ and related numbers}, 
J. reine angew. Math. 342 (1983), 173--196. 


\bibitem{BovdPVa96}
E. Bombieri, A. J. van der Poorten and J. D. Vaaler,
{\it Effective Measures of Irrationality for Cubic Extensions of Number Fields}, 
{ Ann. Scuola Norm. Sup. Pisa Cl. Sci.} 23 (1996), 211--248.





\bibitem{Bu02}
Y. Bugeaud,
{\it Linear forms in two $m$-adic logarithms and applications to Diophantine problems}, 
Compositio Math. 132 (2002), 137--158.



\bibitem{Bu09}
Y. Bugeaud, 
{\it On the convergents to algebraic numbers}. 
In: Analytic number theory, 133--143, Cambridge Univ. Press, Cambridge, 2009. 



\bibitem{Bu18a}
Y. Bugeaud, 
{\it On the digital representation of integers with bounded prime factors}, 
Osaka Math. J. 55 (2018), 315--324. 

\bibitem{Bu18b}
Y. Bugeaud, 
Linear forms in logarithms and applications. 
IRMA Lectures in Mathematics and Theoretical Physics 28, 
European Mathematical Society, Z\"urich, 2018.

\bibitem{Bu20}
Y. Bugeaud, 
{\it Effective simultaneous rational approximation to pairs of real quadratic numbers}, 
Mosc. J. Comb. Number Theory 9 (2020), 353--360. 



\bibitem{Bu21}
Y. Bugeaud, 
{\it On the Zeckendorf representation of smooth numbers},
Moscow Math. J. 21 (2021), 31--42. 


\bibitem{Bu22}
Y. Bugeaud, 
{\it Fractional parts of powers of real algebraic numbers}, 
C. R. Math. Acad. Sci. Paris 360 (2022), 459--466. 

\bibitem{Bu22b}
Y. Bugeaud, 
{\it On effective approximation to quadratic numbers}. 
Preprint.


\bibitem{BuEv17}
Y. Bugeaud and J.-H. Evertse, 
{\it $S$-parts of terms of integer linear recurrence sequences}, 
Mathematika 63 (2017), 840--851. 



\bibitem{BuEvGy18}
Y. Bugeaud, J.-H. Evertse, and K. Gy\H ory, 
{\it $S$-parts of values of of univariate polynomials, binary forms 
and decomposable forms at integral points}, 
Acta Arith. 184 (2018), 151--185. 


\bibitem{BuKa18}
Y. Bugeaud and H. Kaneko, 
{\it On the digital representation of smooth numbers}, 
Math. Proc. Cambridge Philos. Soc.  165 (2018), 533--540. 



\bibitem{BuLa96}
Y. Bugeaud et M. Laurent,
{\it Minoration effective de la distance $p$-adique entre puissances de
nombres alg\'ebriques}, 
J. Number Theory 61 (1996), 311--342.


\bibitem{BuKh23}
Y. Bugeaud and K. D. Nguyen, 
{\it Some arithmetical properties of convergents of algebraic numbers}. 
Preprint. 

\bibitem{CZ04}
P. Corvaja and U. Zannier,
{\it On the rational approximations to the powers of an algebraic number: 
solution of two problems of Mahler and Mend\`es France}, 
Acta Math. 193 (2004), 175--191.


\bibitem{CoZa05}
P. Corvaja and U. Zannier,
{\it $S$-unit points on analytic hypersurfaces}, 
Ann. Sci. \'Ecole Norm. Sup. 38 (2005), 76--92. 



\bibitem{ErMa}
P. Erd\H{o}s and K. Mahler,
{\it Some arithmetical properties of the convergents
of a continued fraction},
J. London Math. Soc. 14 (1939), 12--18.



\bibitem{EvGy15}
J.-H. Evertse and K. Gy\H ory, 
Unit equations in Diophantine number theory. 
Cambridge Stud. Adv. Math. 146, Cambridge Univ. Press, 2015. 

\bibitem{EvGy16}
J. H. Evertse and K. Gy\H{o}ry,
Discriminant equations in Diophantine number theory.
New Mathematical Monographs 32, Cambridge University Press, 2016. 

\bibitem{Fe68}
{N. I. ~Fel'dman},
{\it Improved estimate for a linear form of the logarithms of 
algebraic numbers}, 
{Mat. Sb.} {  77} (1968), 393--406 (in Russian);
English translation in Math. USSR. Sb.  6 (1968) 423--436.

\bibitem{Fe71}
{N. I. ~Fel'dman},
{\it An effective refinement of the exponent in Liouville's theorem}, 
Iz. Akad. Nauk SSSR, Ser. Mat. 35 (1971), 973--990 (in Russian);
English translation in Math. USSR. Izv. 5 (1971) 985--1002. 

\bibitem{Gou06}
N. Gouillon, 
{\it Explicit lower bounds for linear forms in two logarithms}, 
J. Th\'eor. Nombres Bordeaux 18 (2006), 125--146. 


\bibitem{Hu20}
Y. Huang,  
{\it Unit equations on quaternions}, 
Q. J. Math. 71 (2020), 1521--1534. 

\bibitem{Lau94}
M. Laurent, 
{\it Linear forms in two logarithms and interpolation determinants}, 
Acta Arith. 66 (1994), 181--199. 

\bibitem{LMPW87}
J. H. Loxton, M. Mignotte, A. J. van der Poorten, and M. Waldschmidt, 
{\it A lower bound for linear forms in the logarithms of algebraic numbers}, 
C. R. Math. Rep. Acad. Sci. Canada 9 (1987), 119--124. 


\bibitem{Mah53}
K. Mahler,
{\it On the approximation of logarithms of algebraic numbers}, 
Philos. Trans. R. Soc. Lond. Ser. A 245 (1953), 371--398. 


\bibitem{Mah57}
K. Mahler,
{\it On the fractional parts of the powers of a rational number, II},
Mathematika 4 (1957), 122--124.


\bibitem{Mat00}
E. M. Matveev,
{\it An explicit lower bound for a homogeneous rational linear form in 
logarithms of algebraic numbers. II}, Izv.
Ross. Akad. Nauk Ser. Mat. 64 (2000), 125--180.



\bibitem{Ri93}
J. H. Rickert, 
{\it Simultaneous rational approximations and related Diophantine equations}, 
Math. Proc. Cambridge Philos. Soc. 113 (1993),  461--472. 



\bibitem{Rid57}
D. Ridout,
{\it Rational approximations to algebraic numbers},
Mathematika 4 (1957), 125--131.


\bibitem{Rid58}
D. Ridout,  
{\it The p-adic generalization of the Thue–Siegel–Roth theorem}, 
Mathematika 5 (1958), 40--48. 



\bibitem{Schm67}     
W. M. Schmidt, 
{\it On simultaneous approximations of two algebraic numbers by rationals}, 
Acta Math. 119 (1967), 27--50. 



\bibitem{Sho77}
T. N. Shorey,
{\it Some applications of linear forms in logarithms},
S\'eminaire Delange--Pisot--Poitou, 17e ann\'ee: 1975/76. 
Th\'eorie des nombres: Fasc. 1, Exp. No. 3, 8 pp.
Secr\'etariat Math., Paris, 1977. 


\bibitem{Sho83}
T. N. Shorey,
{\it Divisors of convergents of a continued fraction},
J. Number Theory 17 (1983), 127--133.

\bibitem{ShTi86}
T. N. S{horey} and R. T{ijdeman},
Exponential Diophantine equations. 
Cambridge Tracts in Mathematics 87,
Cambridge University Press, Cambridge, 1986.



\bibitem{Spr93}
V. G. Sprind\v zuk,
Classical Diophantine Equations. 
Lecture Notes in Math. 1559,
Springer-Verlag, Berlin, 1993.

\bibitem{Wa93}
M. Waldschmidt,
{\it Minorations de combinaisons 
lin\'eaires de logarithmes de nombres
alg\'ebri\-ques}, Canadian J. Math. 45 (1993), 176--224.


\bibitem{WaLiv}
M. Waldschmidt,
Diophantine Approximation on Linear Algebraic Groups. 
Transcendence Properties of the Exponential Function in Several Variables, 
Grundlehren Math. Wiss. 326, Springer, Berlin, 2000.


\bibitem{Wa03}
M. Waldschmidt,
{\it Linear independence measures for logarithms of algebraic numbers}. 
In: Diophantine approximation (Cetraro, 2000), 250--344, 
Lecture Notes in Math., 1819, Springer, Berlin, 2003. 


\bibitem{Yu07}
K. Yu,
{\it $p$--adic logarithmic forms and group varieties, III},  
Forum Math.  19  (2007),  187--280. 

\bibitem{Zu07}
W. Zudilin, 
{\it A new lower bound for $\| (3/2)^k \|$}, 
J. Th\'eor. Nombres Bordeaux 19 (2007); 311--323. 


\end{thebibliography}
\end{document}